\documentclass[12pt,reqno]{amsart}

\usepackage{amsmath,amsfonts,amsbsy,amsgen,amscd,mathrsfs,amssymb,amsthm}
\usepackage{enumerate,mathtools}

\usepackage[usenames,dvipsnames]{xcolor}
\usepackage[colorlinks=true,citecolor=blue,linkcolor=blue]{hyperref}

\usepackage{tikz,comment}
\usetikzlibrary{shadings}

\newtheorem{theorem}{Theorem}
\newtheorem{lemma}[theorem]{Lemma}

\newtheorem*{defn}{Definition}
\newtheorem*{prop*}{Proposition}
\newtheorem{conj}{Conjecture}
\newtheorem*{conj*}{Conjecture}

\newtheorem*{fact*}{Fact}
\newtheorem{prop}{Proposition}
\newtheorem*{ex*}{Example}

\newtheorem{rem}{Remark}

\DeclareMathOperator{\tr}{tr}
\DeclareMathOperator*{\E}{\mathbb{E}}

\newcommand{\BH}{\textnormal{BH}}
\newcommand{\al}{\alpha}
\newcommand{\eps}{\varepsilon}
\newcommand{\om}{\omega}
\newcommand{\Om}{\Omega}
\newcommand{\supp}{\mathrm{supp}}

\newcommand{\mc}[1]{\mathcal{#1}}

\newcommand{\Lesssim}[1]{\lesssim_{{\textstyle\mathstrut}{#1}}}

\theoremstyle{definition}

\numberwithin{equation}{section} 
\numberwithin{figure}{section}
\numberwithin{table}{section}

\newcommand{\C}{\mathbf{C}}
\newcommand{\bC}{\mathbf{C}}
\newcommand{\T}{\mathbf{T}}

\newcommand{\Z}{\mathbf{Z}}

\newcommand{\bR}{\mathbf{R}}

\newcommand{\Sidon}{\textnormal{Sidon}}

\newcommand{\1}{\mathbf{1}}

\newcommand{\bi}{\mathbf{i}}
\newcommand{\bj}{\mathbf{j}}
\newcommand{\mult}{\textnormal{mult}}
\newcommand{\conv}{\textnormal{conv}}

\begin{document}
	
	\title[Bohnenblust--Hille inequality for cyclic groups]{Bohnenblust--Hille inequality for cyclic groups}
	
	\author{Joseph Slote}
	\address{(J.S.) Department of Computing \& Mathematical Sciences,  California Institute of Technology, Pasadena, CA 91125}
	\email{jslote@caltech.edu}

	\author{Alexander Volberg}
	\address{(A.V.) Department of Mathematics, MSU, 
		East Lansing, MI 48823, USA and Hausdorff Center of Mathematics}
	\email{volberg@math.msu.edu}
	
	\author{Haonan Zhang}
	\address{(H.Z.) Department of Mathematics, University of South Carolina, Columbia, SC, 29208, USA}
	\email{haonanzhangmath@gmail.com}
	
	\begin{abstract}
	For any $K>2$ and the multiplicative cyclic group $\Omega_K$ of order $K$, consider any function $f:\Omega_K^n\to\C$ and its Fourier expansion $f(z)=\sum_{\alpha\in\{0,1,\ldots,K-1\}^n}a_\alpha z^\alpha$, with $d:=\deg(f)$ denoting its degree as a multivariate polynomial.
        We prove a Bohnenblust--Hille (BH) inequality in this setting: the $\ell_{2d/(d+1)}$ norm of the Fourier coefficients of $f$ is bounded by $C(d,K)\|f\|_\infty$ with $C(d,K)$ independent of $n$.
        This is the interpolating case between the now well-understood BH inequalities for functions on the poly-torus ($K =\infty$) and the hypercube ($K=2$) but those extreme cases of $K$ have special properties whose absence for intermediate $K$ prevent a proof by the standard BH framework.
        New techniques are developed exploiting the group structure of $\Omega_K^n$.
        
        By known reductions, the cyclic group BH inequality also entails a noncommutative BH inequality for tensor products of the $K \times K$ complex matrix algebra (or in the language of quantum mechanics, systems of $K$-level qudits). These new BH inequalities generalize several applications in harmonic analysis and statistical learning theory to broader classes of functions and operators.

	\end{abstract}
	
	


	\maketitle
	
	\section{Introduction}
	
	Let $\T^n$ be the $n$-dimensional torus. For any $d\ge 1$, let $\mc{P}_d(\mathbf{T}^n)$ denote the family of analytic polynomials in $n$ complex variables $z:=(z_1, \dots, z_n)\in \T^n$ of degree at most $d$.
	In other words, $P\in \mc{P}_d(\mathbf{T}^n)$ if
	$$
	P(z) =\sum_{|\al|\le d}c_\al z^\al,
	$$
	where $z^\al:=z_1^{\al_1}\dots z_n^{\al_n}$ with
	\begin{equation*}
		\al:=(\al_1,\dots,  \al_n)\in \Z_{\ge 0}^n,\qquad |\al|:=\al_1+\dots +\al_n.
	\end{equation*}
	A classical inequality due to Bohnenblust and Hille  \cite{BH} states that for fixed $d\ge 1$, for any $n\ge 1$ and for any $P\in \mc{P}_d(\mathbf{T}^n)$ as above we have
	\begin{equation}
		\label{BHT}
		\Big(\sum_{|\al|\le d} |c_\al|^{\frac{2d}{d+1}}\Big)^{\frac{d+1}{2d}} \le C(d) \|P\|_{L^\infty(\mathbf{T}^n)}\,.
	\end{equation}
	Here, $C(d)<\infty$ is some constant depending only on $d$ and is independent of $n\ge1$. This family of inequalities is known as the \emph{(polynomial) Bohnenblust--Hille inequalities}, or  \emph{BH inequalities} for short.
	Since its discovery in 1931 \cite{BH}, it has found many applications in areas as diverse as harmonic analysis, functional analysis, analysis of Boolean functions, probability, infinite-dimensional holomorphy, and analytic number theory \cite{DS,DGMS}.
	
	We remark here that in many applications \cite{DFOOS,BPS}, it is important to have good upper bounds on the best constant $C(d)$ (denoted by $\BH^{\le d}_{\T}$) in Bohnenblust--Hille inequalities \eqref{BHT}. 
	The best known bound is $\BH^{\le d}_{\T}\le C^{\sqrt{d\log d}}$ \cite{BPS} for some universal constant $C>0$, which is of at most sub-exponential growth in $d$. 
	We refer to the monograph \cite{DGMS} for more information.
	
	\medskip
	
	Bohnenblust--Hille inequalities can be extended to more-general contexts; for example, to the vector-valued case in relation to the geometry of Banach spaces \cite{DS,DGMS}.
	How about groups other than $\T^n$?
	This is a natural question in harmonic analysis, and the answer was known \cite{Blei,DMP} for products of the cyclic group of order 2, or the \emph{Boolean cube} $\Omega_2^n:=\{-1,1\}^n$.
	Recall that any function $f:\Omega_2^n\to \mathbf{C}$ has the Fourier--Walsh expansion:
	\begin{equation*}
		f(x)=\sum_{A\subset [n]}\widehat{f}(A)x^A,\qquad x=(x_1,\dots, x_n)\in \Omega_2^n
	\end{equation*}
	where for $A\subset [n]:=\{1,2,\dots, n\}$, $|A|$ denotes the cardinality of $A$, and $x^A:=\prod_{j\in A}x_j$.
	The degree of $f$ is defined as the maximum of $|A|$ for which $\widehat{f}(A)\ne 0$.
	In \cite{DMP}, Defant, Masty\l o, and P\'erez proved that Bohnenblust--Hille inequalities hold for $\Omega_2^n$:
	\begin{equation}\label{ineq:bh boolean}
		\left(\sum_{|A|\le d}|\widehat{f}(A)|^{\frac{2d}{d+1}}\right)^{\frac{d+1}{2d}}\le C(d)\|f\|_{L^\infty (\Omega_2^n)}.
	\end{equation}
	Moreover, they showed that the best possible constant $C(d)>0$, denoted by $\BH^{\le d}_{\Om_2}$, satisfies $\BH^{\le d}_{\Om_2}\le C^{\sqrt{d\log d}}$, which is similar to $\T$. 
	
	Motivations for studying BH inequalities on other groups also come from applications in theoretical computer science.
	The hypercube Bohnenblust--Hille inequalities \eqref{ineq:bh boolean} have recently found unexpected applications to learning theory \cite{EI22} and it is natural to ask about their generalization.
    Specifically, with the help of \eqref{ineq:bh boolean}, Eskenazis and Ivanisvili \cite{EI22} proved that a logarithmic number of random samples suffices to learn bounded low-degree functions $f:\Om_2^n\to [-1,1]$, exponentially improving the dimension-dependence of the sample complexity in previous work (in fact, logarithmic dependence is sharp \cite{EIS}).
    
    This motivated the study of Bohnenblust--Hille inequalities in the quantum realm as well, with a view towards learning quantum observables.
    In \cite{RWZ}, a conjecture of Bohnenblust--Hille inequalities was made for qubit systems (the $2\times 2$ matrix case) and it was swiftly resolved in \cite{CHP} and \cite{VZ22}. In an earlier version of our previous work \cite{SVZ}, we studied the Bohnenblust--Hille inequalities for qudit systems (the $K\times K$ matrix case) that are more general than qubit systems.
    It turns out that we may reduce the problem of qudit Bohnenblust--Hille inequalities with respect to the Heisenberg--Weyl basis to classical Bohnenblust--Hille inequalities for cyclic groups of order $K$ (at least for prime $K$ at the time of writing), which was until that point unstudied.
	
	Therefore, the analogs of Bohnenblust--Hille inequalities for products of general cyclic groups of order $K>2 $ are of importance not only for Fourier analysis but also for application areas such as statistical learning theory.
	
		\medskip
		
		In the present article we prove that Bohnenblust--Hille inequalities hold for all general cyclic groups $\Omega_K^n$ with dimension-free constants.
	More precisely, fix $K>2$ and let $\Omega_K=\{1,\omega,\dots, \omega^{K-1}\}$ be the multiplicative cyclic group of order $K$ with $\omega=e^{\frac{2\pi i}{K}}$. For any $n\ge 1$ and any $f:\Omega_K^n\to \mathbf{C}$ with Fourier expansion 
	\begin{equation}\label{eq:fourier expansion cyclic}
		f(z)=\sum_{\alpha}\widehat{f}(\alpha)z^{\alpha}=\sum_{\alpha}a_{\alpha} z^{\alpha},\qquad z=(z_1,\dots, z_n)\in \Omega_K^n
	\end{equation}
	where $z^\al:=z_1^{\al_1}\dots z_n^{\al_n}$ with
	\begin{equation*}
		\al:=(\al_1,\dots,  \al_n)\in \{0,1,\dots, K-1\}^n,\qquad |\al|:=\al_1+\dots +\al_n.
	\end{equation*}
	Then $f:\Omega_K^n\to \mathbf{C}$ is said to be \emph{of degree at most $d$} if $a_{\alpha}=0$ whenever $|\alpha|>d$.
	We fix this definition of degree throughout the paper, but remark that there are other definitions of $|\cdot|$ in the literature that induce different notions of degree.
	
	In the following, we denote by $\|f\|_X$  the supremum norm of some function $f$ over some set $X$. We write $\|\widehat{f}\|_p$ for the $\ell_p$-norm of the Fourier coefficients for $f$ in \eqref{eq:fourier expansion cyclic},
	\begin{equation*}
	\|\widehat{f}\|_p:=\left(\sum_{\alpha}|\widehat{f}(\alpha)|^p\right)^{1/p}=\left(\sum_{\alpha}|a_{\alpha}|^p\right)^{1/p}.
\end{equation*}		
	
	Our main result is the following:
	
	\begin{theorem}[Cyclic BH inequalities]\label{thm:bh cyclic}
		Fix $d\ge 1$ and $K>2$. Then there exists $C(d,K)>0$ depending only on $d$ and $K$ such that for any $n\ge 1$ and for any 
		\begin{equation*}
			f(z)=\sum_{|\alpha|\le d}a_{\alpha}z^{\alpha},\qquad z\in \Omega_K^n
		\end{equation*}
		that is of degree at most $d$, we have 
		\begin{equation}\label{ineq:bh cyclic}
			\|\widehat{f}\|_{\frac{2d}{d+1}}=\left(\sum_{|\alpha|\le d}|a_{\alpha}|^{\frac{2d}{d+1}}\right)^{\frac{d+1}{2d}}\le C(d,K)\|f\|_{\Om_K^n}.
		\end{equation}
		If $K$ is prime, we may choose $C(d,K)=C_K^{d^2}$ for some constant $C_K$ depending only on $K$.
	\end{theorem}
	
	Partial results were obtained in an earlier version of our work \cite{SVZ}, where we proved \eqref{ineq:bh cyclic} under an additional constraint that $f$ has \emph{individual degree} less than $K/2$ with constants $C(d,K)=C_K^d$.
	By individual degree we mean the maximum degree of $f$ when viewing it as a univariate polynomial in each of the variables $z_1,\dots, z_n$ (so \emph{e.g.,} $z_1^2z_2^3$ has individual degree $3$).
	This partial result is recalled in Appendix \ref{sect:partial} to compare different approaches to the same problem.
	
	\medskip
	
	The proof of Defant, Masty\l o, and P\'erez of Bohnenblust--Hille inequalities \eqref{ineq:bh boolean} for $\Omega_2^n$ follows the similar scheme of the proof for $\T^n$ case in \cite{BPS}, with certain technical adjustments subject to accommodate the special structure of $\Omega_2^n$.
	Although most parts of this scheme carry over for general cyclic groups $\Omega_K^n$, one crucial polarization step does not apply to $2<K<\infty$ (it works for both $\Omega_2^n$ and $\Omega_\infty^n:=\T^n$), whence it does not directly yield the  Bohnenblust--Hille inequalities \eqref{ineq:bh cyclic}.
	As we explain, a certain ``BH-like'' inequality can still be obtained for $2<K<\infty$ following this scheme by altering the domain of $f$, but this seems to erase the discrete lattice structure of $\Om_K^n$ and for example renders it unusable for applications (such as deriving the results in the quantum counterpart). In Section \ref{sect:obstacle} we give a brief explanation and in Appendix \ref{sect:appendix} we detail this difference between the previous work (on $\T^n$ and $\Omega_2^n$) and the case of $\Omega_K^n,2<K<\infty$.

	Section \ref{sec:the-proof} is devoted to the proof of Theorem \ref{thm:bh cyclic} in full generality, where we circumvent the challenges described in Section \ref{sect:obstacle} as well as Appendix \ref{sect:appendix}.
    Different from the classical argument, we shall prove Bohnenblust--Hille inequalities for general cyclic groups \eqref{ineq:bh cyclic} by reducing the problem to $\Om_2^n$ \eqref{ineq:bh boolean}.
	The reduction is technical.
	One technique gives good constants when the order $K$ is prime, while the other technique yields the result for all $K>2$ but the constant becomes worse. Our proof strategy in this paper is also different from that of an earlier version of our previous work \cite{SVZ} establishing the partial results aforementioned. 

	In Section \ref{sect:applications}, we discuss several applications, including results about Sidon constants, juntas, learning bounded low-degree functions and a noncommutative Bohnenblust--Hille inequality for  matrices whose degree is low in the Heisenberg--Weyl basis. 
	
	\medskip
	
	In the following, the implicit constants $C_K, C(d,K)...$ may differ from line to line.
	
	\subsection*{Acknowledgments}	J.S. is supported by Chris Umans' Simons Institute Investigator Grant. A.V. is supported by NSF DMS-1900286, DMS-2154402 and by Hausdorff Center for Mathematics.

This work was started  while all three authors were in residence at the Institute for Computational and Experimental Research in Mathematics in Providence, RI, during the Harmonic Analysis and Convexity program. It is partially supported by NSF  DMS-1929284. The authors would like to thank the anonymous referee for the careful reading and valuable remarks.
	
	\section{The obstacle: a generalized maximum modulus principle}
	\label{sect:obstacle}
In this section we briefly explain the difficulties of adapting the proof of the Bohnenblust--Hille inequalities for $\T^n$ and $\Omega_2^n$ in \cite{BPS,DMP} to $\Omega_K^n$ for $K>2$.
We do find that retracing the old proof can yield interesting estimates (see \eqref{ineq:convex hull bh} below), but such estimates will not be used in our proof of Theorem \ref{thm:bh cyclic}.
We therefore defer a lengthier discussion of the \cite{BPS,DMP} approach to Appendix \ref{sect:appendix}.

The main obstacle can be explained compactly.
It seems that the best that one can deduce from the old argument (see Appendix \ref{sect:appendix}) is the following convex hull version Bohnenblust--Hille inequalities for cyclic groups $\Omega_K^n,K>2$:
	\begin{equation}\label{ineq:convex hull bh}
\left(		\sum_{|\alpha|\le d}|a_{\alpha}|^{\frac{2d}{d+1}}\right)^{\frac{d+1}{2d}}
\le C(d,K)\|f\|_{\conv(\Omega_K)^n}
	\end{equation}
	for all $f(z)=\sum_{|\alpha|\le d}a_{\alpha}z^\alpha$ on $\Omega_K^n
$. Here, $f$ on the right-hand side is understood as its natural extension via \eqref{eq:fourier expansion cyclic}  to $\conv(\Omega_K)^n$ as an analytic polynomial of degree at most $d$ and individual degree at most $K-1$. 

In fact, we remark there is a separate direct argument that arrives at the same inequality \eqref{ineq:convex hull bh}, albeit with a worse constant.
Take $c=c_K=\cos(\frac{\pi}{K})^{-1}\in (1,2]$, which is the smallest $c>0$ such that $c^{-1}\T\subset \conv(\Omega_K)$. Since $|\alpha|\le d$, we have
\begin{align*}
	\left(		\sum_{|\alpha|\le d}|a_{\alpha}|^{\frac{2d}{d+1}}\right)^{\frac{d+1}{2d}}
	\le c_K^d	\left(		\sum_{|\alpha|\le d}\left|\frac{a_\alpha}{c_K^{|\alpha|}}\right|^{\frac{2d}{d+1}}\right)^{\frac{d+1}{2d}}.
\end{align*}
Applying Bohnenblust--Hille inequalities for  $\T^n$ to 
\begin{equation*}
	z\mapsto f(z/c_K)=\sum_{|\alpha|\le d}\frac{a_{\alpha}}{c_K^{|\alpha|}}z^\alpha,
\end{equation*} 
one has 
\begin{equation*}
	\left(		\sum_{|\alpha|\le d}\left|\frac{a_\alpha}{c_K^{|\alpha|}}\right|^{\frac{2d}{d+1}}\right)^{\frac{d+1}{2d}}
	\le \BH^{\le d}_{\T}\sup_{z\in \T^n}\left|f(z/c_K)\right|
		\le \BH^{\le d}_{\T}\sup_{z\in (c_K^{-1}\T)^n}\left|f(z)\right|,
\end{equation*}
and the right-hand side is bounded from above by 
\begin{equation*}
	\BH^{\le d}_{\T}\sup_{z\in \conv(\Omega_K)^n}\left|f(z)\right|,
\end{equation*}
according to our choice of $c_K$. This concludes the proof of \eqref{ineq:convex hull bh} with $C(d,K)=c_K^d\BH^{\le d}_{\T}$ that is of exponential growth in $d$.

	\medskip
	
	So to prove the Bohnenblust--Hille inequality \eqref{ineq:bh cyclic} for cyclic groups $\Omega_K^n$, we may reduce to showing
	\begin{equation}\label{ineq:mmp}
		\sup_{z\in \conv(\Omega_K)^n}\left|f(z)\right|
		\le C(d,K)\sup_{z\in \Omega_K^n}\left|f(z)\right|.
	\end{equation}	
	For $K=\infty$ or $K=2$ this is trivial with constant $C(d)=1$: in the former case, we have \eqref{ineq:mmp} by the maximum modulus principle; in the latter case, $f:\Omega_2^n\to \mathbf{C}$ is always multi-affine and convexity suffices. However, in the general case $2<K<\infty$, \eqref{ineq:mmp} seems highly non-trivial.
	This leads to the following conjecture:
	
	\begin{conj}[Generalized Maximum Modulus Principle]
		\label{conjectuer}
		Fix $K>2$ and $d\ge 1$. Then there exists $C(d,K)>0$ such that for all $n\ge 1$ and for all $f:\Omega_K^n\to \mathbf{C}$ that is of degree at most $d$ and of individual degree at most $K-1$:
		\begin{equation*}
			f(z)=\sum_{\alpha\in \{0,1,\dots, K-1\}^n:|\alpha|\le d}a_{\alpha}z^{\alpha},
		\end{equation*}
		we have 
		\begin{equation*}
			\max_{z\in \conv(\Omega_K)^n}|f(z)|\le 
			C(d,K)\max_{z\in \Omega_K^n}|f(z)|.
		\end{equation*}
	\end{conj}
	
	Our proof of Theorem \ref{thm:bh cyclic} below fully circumvents Conjecture \ref{conjectuer}.
	The full positive answer to Conjecture \ref{conjectuer} was later obtained in \cite{SVZgmp}, and in Appendix \ref{sect:partial} we provide a partial answer attempt that was present in an earlier version of \cite{SVZ} before our later work \cite{SVZgmp}.

	\begin{rem}
	\label{2223}
	It is important to note that in certain applications (see Section \ref{sect:applications}), the convex hull version \eqref{ineq:convex hull bh} cannot be used and we really do need \eqref{ineq:bh cyclic}.
	\end{rem}

	\section{Proof of Theorem \ref{thm:bh cyclic}}
	\label{sec:the-proof}
	
	In the torus  $\T^n$ case, the BH inequalities \eqref{BHT} can be reduced to \emph{homogeneous polynomials}, \emph{i.e.}, linear combinations of monomials $z^\alpha$ with same $|\alpha|$.
	However, it is not clear how to do this for $\Om_2^n$. The proof of \eqref{ineq:bh boolean} for non-homogeneous polynomials is more involved; see \cite{DMP} for details.
    In the setting of Theorem \ref{thm:bh cyclic}, it is similarly unclear whether we may reduce the proof of BH inequalities for $\Omega^n_K$ \eqref{ineq:bh cyclic} to homogeneous polynomials only.
	Ultimately we succeed by reducing to a different notion of homogeneity, given in terms of the size of the support of $\alpha$. Write $|\supp(\alpha)|:=|\{j:\alpha_j\neq 0\}|$. Throughout this section, we fix $K\ge 3$.
	
	\begin{defn}[Support-homogeneous polynomials]
		A polynomial $f:\Om_K^n\to\C$ is $\ell$-\emph{support-homogeneous} if it can be written as
		\[f(z) = \sum_{\alpha:|\supp(\alpha)|=\ell}a_\alpha z^\alpha\,.\]
	\end{defn}
	
	We will employ a certain operator that removes monomials whose support sizes are not maximal and alters the coefficients of the remaining terms (those of maximal support size) in a controlled way.
	
	\begin{defn}[Maximal support pseudo-projection]
		Fix $\xi\in \Om_K\backslash\{1\}$.
		For any multi-index $\alpha\in\{0,1,\ldots, K-1\}^n$ we define the factor
		\begin{equation}
			\tau_\alpha^{(\xi)} =\prod_{j: \alpha_j\neq 0}(1-\xi^{\alpha_j})\,.
			\label{eq:tau}
		\end{equation}
		For any polynomial with maximal support size $\ell\ge 0$,
		$$f(z)=\sum_{|\supp(\alpha)|\le \ell}a_\alpha z^\alpha,$$
		on $\Omega_K^n$, we define $\mathfrak{D}_{\xi} f:\Omega_K^n\times\Omega_2^n\to \C$ via
		\[\mathfrak{D}_{\xi}f(z,x) =\sum_{|\supp(\alpha)|= \ell}\tau_\alpha^{(\xi)}\,a_\alpha z^\alpha x^{\supp(\alpha)}\,.\]
	\end{defn}
	
	Denote by $\vec{1}=(1,\dots, 1)$ the vector in $\Omega_2^n$ that has all entries $1$.  Note that $\mathfrak{D}_{\xi}f(\cdot,\vec 1)$ is the $\ell$-support-homogeneous part of $f$, except where the coefficients $a_\alpha$ have picked up the factor $\tau_\alpha^{(\xi)}$.
	And we note the relationships among the $\tau_\alpha^{(\xi)}$'s can be quite intricate; while in general they are different for distinct $\alpha$'s, this is not always true:
	consider the case of $K=3$ and the two monomials
	\begin{equation}
		\label{ex:same-taus}
		z^{\beta}:=z_1^2 z_2z_3z_4z_5z_6z_7z_8,\quad z^{\beta'}:=z_1^2 z_2^2 z_3^2 z_4^2 z_5^2 z_6^2 z_7^2z_8\,.
	\end{equation}
	Then
	\[
	\tau_{\beta}^{(\omega)}\;=\;(1-\omega)^7(1-\omega^2)\;=\;(1-\omega)(1-\omega^2)^7\;=\;\tau^{(\omega)}_{\beta'},
	\]
	which follows from the identity $(1-\omega)^6=(1-\omega^2)^6$ for $\omega=e^{\frac{2\pi i}{3}}$.
	
	\medskip
	
	The somewhat technical definition of $\mathfrak{D}_{\xi}$ is motivated by the proof of its key property, namely that it is bounded from $L^\infty(\Omega_K^n)$ to $L^\infty({\Om_K^n\times\Om_2^n})$ with a dimension-free constant when restricted to low-degree (actually low support size) polynomials. Before discussing this, let us recall the following result on $\Omega_2^n$:
	
	\begin{prop}\cite[Lemma 1 (iv)]{DMP}\label{prop:riesz proj}
		Consider any polynomial $g:\Omega^n_2\to \mathbf{C}$ of degree at most $d>0$. Suppose that $g_m$ is the $m$-homogeneous part of $g$ for any $m\le d$. Then 
		\begin{equation}
			\|g_m\|_{\Omega_2^n}\le (1+\sqrt{2})^d \|g\|_{\Omega_2^n}.
		\end{equation}
	\end{prop}
	
	\begin{prop}[Dimension-free boundedness of $\mathfrak{D}_{\xi}$]
		
		Let $f:\Omega_K^n\to \mathbf{C}$ be a polynomial of degree at most $d$ and $\ell$ be the maximum support size of monomials in $f$.
		Then for all $\xi \in \Omega_K\backslash\{1\}$,
		\begin{equation}
			\label{ineq:main ingredient}
			\|\mathfrak{D}_{\xi}f\|_{\Om_K^n\times\Om_2^n}\le (2+2\sqrt{2})^{\ell} \|f\|_{\Omega_K^n}.
		\end{equation}
	\end{prop}
	\begin{proof}
		Consider the operator $G_{\xi}$:
		\begin{equation*}
			G_{\xi}(f)(x)
			=f\left(\frac{1+\xi}{2}+\frac{1-\xi}{2}x_1,\dots, \frac{1+\xi}{2}+\frac{1-\xi}{2}x_n\right),\qquad x\in\Omega_2^n
		\end{equation*}
		that maps any function $f:\{1,\xi\}^n\subset \Om_K^n\to \mathbf{C}$ to a function $G_{\xi}(f):\Omega_2^n\to \mathbf{C}$.
		Then by definition
		\begin{equation}\label{ineq:G compare K prime}
			\|f\|_{\Om_K^n}\ge 	\|f\|_{\{1,\xi\}^n}\ge\|G_{\xi}(f)\|_{\Omega_2^n}. 
		\end{equation}
		Fix $m\le \ell$. For any $\alpha$ we denote 
		\begin{equation*}
			m_k(\alpha):=|\{j:\alpha_j=k\}|,\qquad 0\le k\le K-1.
		\end{equation*}
		Then for $\alpha$ with $|\supp(\alpha)|=m$, we have
		\begin{equation}
			m_1(\alpha)+\cdots+m_{K-1}(\alpha)=|\supp(\alpha)|=m.
		\end{equation}
		For $z\in \{1,\xi\}^n$ with $z_j=\frac{1+\xi}{2}+\frac{1-\xi}{2}x_j,$ $x_j=\pm 1$, note that 
		\begin{equation*}
			z_j^{\alpha_j}=\left(\frac{1+\xi}{2}+\frac{1-\xi}{2}x_j \right)^{\alpha_j}=\frac{1+\xi^{\alpha_j}}{2}+\frac{1-\xi^{\alpha_j}}{2}x_j\,.
		\end{equation*}
		So for any $A\subset [n]$ with $|A|=m$, and for each $\alpha$ with $\supp(\alpha)=A$, we have for $z\in \{1,\xi\}^n$:
		\begin{align*}
			z^{\alpha}
			=&\prod_{j:\alpha_j\neq 0}z_j^{\alpha_j}\\
			=&\prod_{j:\alpha_j\neq 0}\left(\frac{1+\xi^{\alpha_j}}{2}+\frac{1-\xi^{\alpha_j}}{2}x_j\right)\\
			=&\prod_{j:\alpha_j\neq 0}\left(\frac{1-\xi^{\alpha_j}}{2}\right)\cdot x^A+\cdots\\
			=&2^{-m}\tau_\alpha^{(\xi)} x^A+\cdots
		\end{align*}
		where $x^A$ is of degree $|A|=m$ while $\cdots$ is of degree $<m$. Then for $f(z)=\sum_{|\supp(\alpha)|\le \ell}a_{\alpha}z^\alpha$ we have
		\begin{align*}
			G_{\xi}(f)(x)
			=\sum_{m\le \ell}\frac{1}{2^m}\sum_{|A|=m}\left(\sum_{ \supp(\alpha)=A} a_{\alpha}\tau_\alpha^{(\xi)}\right)
			x^A+\cdots,\qquad x\in \Omega_2^n.
		\end{align*}
		Again, for each $m\le \ell$, $\cdots$ is some polynomial of degree $<m$. So $G_{\xi}(f)$ is of degree $\le \ell$ and the $\ell$-homogeneous part is nothing but 
		\begin{equation*}
			\frac{1}{2^\ell}\sum_{|A|=\ell}\left(\sum_{ \supp(\alpha)=A} \tau_\alpha^{(\xi)}a_{\alpha}\right)
			x^A.
		\end{equation*}
		Consider the projection operator $Q$ that maps any polynomial on $\Omega_2^n$ onto its highest level homogeneous part; \emph{i.e.}, for any polynomial $g:\Omega_2^n\to \mathbf{C}$ with $\deg(g)=m$ we denote $Q(g)$ its $m$-homogeneous part.
		Then we just showed that 
		\begin{align*}
			Q(G_{\xi}(f))(x)	
			=\frac{1}{2^\ell}\sum_{|A|=\ell}\left(\sum_{ \supp(\alpha)=A} \tau_\alpha^{(\xi)}a_{\alpha}\right)
			x^A.
		\end{align*}
		According to Proposition \ref{prop:riesz proj} and \eqref{ineq:G compare K prime}, we have
		\begin{equation*}
			\|Q(G_{\xi}(f))\|_{\Omega_2^n}\le (1+\sqrt{2})^{\ell}\|G_{\xi}(f)\|_{\Omega_2^n}
			\le (1+\sqrt{2})^{\ell} \|f\|_{\Omega_K^n}
		\end{equation*}
		and thus 
		\begin{align*}
			\left\| \sum_{|A|=\ell}\left(\sum_{ \supp(\alpha)=A} \tau_\alpha^{(\xi)} a_{\alpha}\right)
			x^A\right\|_{\Omega_2^n}
			\le(2+2\sqrt{2})^{\ell}\|f\|_{\Omega_K^n}.
		\end{align*}
		
		The left-hand side is almost $\mathfrak{D}_{\xi}f$.
		Observe that $\Omega_K^n$ is a group, so we have 
		\begin{equation*}
			\sup_{z,\zeta\in\Om_K^n}\left|\sum_{\alpha}a_{\alpha}z^{\alpha}\zeta^{\alpha}\right|
			=\sup_{z\in\Om_K^n}\left|\sum_{\alpha}a_{\alpha}z^{\alpha}\right|.
		\end{equation*}
		Thus actually we have shown
		\begin{equation}
			\sup_{z\in\Om_K^n,x\in \Omega_2^n}\left| \sum_{|A|=\ell}\left(\sum_{ \supp(\alpha)=A} \tau^{(\xi)}_\alpha a_{\alpha}z^{\alpha}\right)
			x^A\right|	\le (2+2\sqrt{2})^{\ell} \|f\|_{\Omega_K^n},
		\end{equation}
		which is exactly \eqref{ineq:main ingredient}.
	\end{proof}
	
	The second set of variables in $\mathfrak{D}_{\xi}f$ are key to proving the support-homogeneous case of the cyclic BH inequalities.
	
	\begin{lemma}[Support-homogeneous cyclic BH inequalities]
		\label{lem:support-homo-BH}
		Let $f_\ell:\Om_K^n\to\C$ be an $\ell$-support-homogeneous polynomial of degree at most $d$.
		Then
		\[\|\widehat{f_\ell}\|_{\frac{2d}{d+1}}\leq C_{K,d}\|f_\ell\|_{\Om_K^n}.\]
		Moreover, we may take $C_{K,d}$ to be $C_K^d$ for $C_K$ independent of $d$ and $n$.
	\end{lemma}
	\begin{proof}
		Let $f_\ell(z) = \sum_{|\supp(\alpha)|=\ell}a_\alpha z^\alpha$. For any $z\in \Omega_K^n$,
		apply the BH inequalities \eqref{ineq:bh boolean} for $\Omega_2^n$ to $x\mapsto \mathfrak{D}_{\omega}f_\ell(z,x)$ and  \eqref{ineq:main ingredient} for $(f;\xi)=(f_\ell;\omega)$ to obtain
		\begin{equation*}
			\sum_{|A|=\ell}\left|\sum_{ \supp(\alpha)=A} \tau^{(\omega)}_\alpha a_{\alpha}z^{\alpha}\right|^{\frac{2d}{d+1}}
			\le\left[ \BH^{\le d}_{\Om_2}(2+2\sqrt{2})^{\ell} \|f_\ell\|_{\Omega_K^n}\right]^{\frac{2d}{d+1}}.
		\end{equation*}
		Taking the expectation with respect to the Haar measure over $z\sim \Omega_K^n$, we get
		\begin{equation}\label{ineq:expectation}
			\sum_{|A|=\ell}\E_{z\sim  \Omega_K^n}\left|\sum_{ \supp(\alpha)=A} \tau^{(\omega)}_\alpha a_{\alpha}z^{\alpha}\right|^{\frac{2d}{d+1}}
			\le\left[ \BH^{\le d}_{\Om_2}(2+2\sqrt{2})^{\ell} \|f_\ell\|_{\Omega_K^n}\right]^{\frac{2d}{d+1}}.
		\end{equation}
		Note that for fixed $A\subset[n]$ with $|A|=\ell$ and for each $\alpha$ with $\supp(\alpha)=A$
		\begin{align*}
			\big|\tau^{(\omega)}_\alpha a_{\alpha}\big|^{\frac{2d}{d+1}}
			\le  &\left(\E_{z\sim  \Omega_K^n}\left|\sum_{ \supp(\alpha)=A} \tau^{(\omega)}_\alpha a_{\alpha}z^{\alpha}\right|\right)^{\frac{2d}{d+1}}\\
			\le& \E_{z\sim  \Omega_K^n}\left|\sum_{ \supp(\alpha)=A} \tau^{(\omega)}_\alpha a_{\alpha}z^{\alpha}\right|^{\frac{2d}{d+1}},
		\end{align*}
		where the first inequality uses Hausdorff--Young inequality and the second inequality follows from H\"older's inequality. The number of all such $\alpha$'s is bounded by (recalling $A$ is fixed)
		\begin{equation*}
			|\{\al: \supp(\al) =A\}| \le (K-1)^\ell\le (K-1)^d,
		\end{equation*}
		thus 
		\begin{align*}
			\sum_{ \supp(\alpha)=A}|\tau^{(\omega)}_\alpha|^{\frac{2d}{d+1}}\cdot |a_{\alpha}|^{\frac{2d}{d+1}}
			\;\le\; (K-1)^d\E_{z\sim  \Omega_K^n}\left|\sum_{ \supp(\alpha)=A} \tau^{(\omega)}_\alpha  a_{\alpha}z^{\alpha}\right|^{\frac{2d}{d+1}}\,.
		\end{align*}
		This, together with  \eqref{ineq:expectation} and the fact that $\BH^{\le d}_{\Omega_2}\le C^d$, yields
		\begin{align*}
			\left(\sum_{|A|=\ell}\sum_{ \supp(\alpha)=A}	|\tau^{(\omega)}_\alpha|^{\frac{2d}{d+1}}\cdot |a_{\alpha}|^{\frac{2d}{d+1}}\right)^{\frac{d+1}{2d}}
			\le C_K^d \|f_\ell\|_{\Omega_K^n},
		\end{align*}
		for some $C_K>0$. Now we bound $|\tau^{(\omega)}_\alpha|$ from below with
		\begin{equation*}
			\big|\tau^{(\omega)}_\alpha\big|=\prod_{j:\alpha_j\neq 0}|1-\omega^{\alpha_j}|\ge |1-\omega|^{\ell}=(2\sin(\pi/K))^\ell\ge (\sin(\pi/K))^d>0\,,
		\end{equation*}
		and therefore, 
		\begin{align*}
			\left(\sum_{|A|=\ell}\sum_{ \supp(\alpha)=A}	|a_{\alpha}|^{\frac{2d}{d+1}}\right)^{\frac{d+1}{2d}}
			\le\left( \frac{C_K}{\sin(\pi/K)}\right)^d \|f_\ell\|_{\Omega_K^n}
			=C_K^d\|f_\ell\|_{\Omega_K^n}. \tag*{\qedhere}
		\end{align*}
	\end{proof}
	
	With the support-homogeneous Cyclic BH inequality proved, it remains to reduce the proof of BH inequalities \eqref{ineq:bh cyclic} from the general case to the support-homogeneous case.
	This is done by proving the following lemma as an analog of Proposition \ref{prop:riesz proj} for general cyclic groups $\Omega_K^n$.
	
	\begin{lemma}[Splitting Lemma]
		\label{lem:splitting-lemma}
		Let $f:\Om_K^n\to \C$ be a polynomial of degree at most $d$ and for $0\le j\le d$ let $f_j$ be its $j$-support-homogeneous part.
		Then for all $0\le j\le d$, 
		\begin{equation}\label{ineq:splitting lemma}
		\|f_j\|_{\Om_K^n}\leq C_{K,d}\|f\|_{\Om_K^n}
		\end{equation}
		where $C_{K,d}$ is a constant independent of $n$.
		Moreover, for $K$ prime we may take $C_{K,d}$ to be $C_K^{d^2}$.
	\end{lemma}
	
	We pause to note Lemmas \ref{lem:support-homo-BH} and \ref{lem:splitting-lemma} together immediately give the full BH inequality:
	
	
	\begin{proof}[Proof of Theorem \ref{thm:bh cyclic}]
		 For $0\le j\le d$ let $f_j$ be the $j$-support-homogeneous part of $f$. The triangle inequality and our lemmas give
		\begin{align*}
		\|\widehat{f}\|_{\frac{2d}{d+1}}\;\leq\; \sum_{0\le j\le d}\|\widehat{f_j}\|_{\frac{2d}{d+1}}\overset{\text{Lemma \ref{lem:support-homo-BH}}}{\Lesssim{K,d}}\sum_{0\le j\le d}\|f_j\|_{\Omega_K^n}\overset{\text{Lemma \ref{lem:splitting-lemma}}}{\Lesssim{K,d}}\|f\|_{\Omega_K^n}\,.
		\end{align*}
		Estimates of $C_{K,d}$ easily follow from the lemma statements.
	\end{proof}
	
	There are two proofs of Lemma \ref{lem:splitting-lemma}, one for prime $K>2$ achieving explicit constants and one for all $K>2$ with implicit constants.
	
	\medskip
	
	They both employ $\mathfrak{D}_{\xi}f(\cdot,\vec1)$, so for simplicity of notation we shall denote $\mathfrak{D}_{\xi}f(\cdot,\vec1)$ by $\mathfrak{D}_{\xi}f(\cdot):\Om_K^n\to\C$ unless otherwise stated.
    Concretely, in what follows we have
	\begin{equation*}
		\mathfrak{D}_{\xi}f(z):=\mathfrak{D}_{\xi}f(z,\vec1)= \sum_{|\supp(\alpha)|=\ell}\tau_\alpha^{(\xi)}\,a_\alpha z^\alpha \,
	\end{equation*}
	for polynomials with largest support size $\ell$: 
	$$
	f(z)=\sum_{|\supp(\alpha)|\le \ell}a_\alpha z^\alpha, \qquad a_{\alpha}\neq 0 \textnormal{ for some } |\supp(\alpha)|= \ell.
	$$
	
	\begin{proof}[Proof of Lemma \ref{lem:splitting-lemma} for prime $K$]
		Let $\ell\le d$ be the largest support size of monomials in $f$. For any $m\le \ell$, we denote by $f_m$ the $m$-support homogeneous part of $f$. Then we shall prove the lemma for $f_\ell$ first.
		
		To show \eqref{ineq:splitting lemma} for $f_\ell$, recall that \eqref{ineq:main ingredient} gives
		\begin{align*}
			\|\mathfrak{D}_{\xi}f\|_{\Om_K^n}=\left\| \sum_{ |\supp(\alpha)|=\ell} \tau_\alpha^{(\xi)}a_{\alpha}z^{\alpha}\right\|_{\Omega_K^n}
			\le (2+2\sqrt{2})^{\ell}\|f\|_{\Omega_K^n}\,.
		\end{align*}
		We would like to replace the polynomial on the left-hand side with $f_\ell$, and the main issue is that the factors
		\begin{equation*}
			\tau_\alpha^{(\xi)}=\prod_{j:\alpha_j\neq 0}\left(1-\xi^{\alpha_j}\right),
		\end{equation*}
		may differ for different $\alpha$'s under consideration as we discussed before. To overcome this difficulty, we rotate $\xi$ over repeated applications of $\mathfrak{D}_{\xi}$.
		This will lead to an accumulated factor that is constant across all monomials in $f_\ell$.
		Begin by considering $\mathfrak{D}_{\omega}f$ to obtain 
		\begin{align*}
			\left\| \sum_{ |\supp(\alpha)|=\ell} \tau_\alpha^{(\omega)}a_{\alpha} z^{\alpha}\right\|_{\Omega_K^n}
			\le (2+2\sqrt{2})^{\ell}\|f\|_{\Omega_K^n}.
		\end{align*}
		Then we apply $\mathfrak{D}_{\omega^2}$ to $\mathfrak{D}_{\omega}f$ to obtain 
		\begin{align*}
			\left\| \sum_{ |\supp(\alpha)|=\ell} \tau_\alpha^{(\omega)}\tau_\alpha^{(\omega^2)}a_{\alpha} z^{\alpha}\right\|_{\Omega_K^n}
			\le &(2+2\sqrt{2})^{\ell}\|\mathfrak{D}_{\omega}f\|_{\Omega_K^n}\\
			\le &(2+2\sqrt{2})^{2\ell}\|f\|_{\Omega_K^n}.
		\end{align*}
		We continue iteratively applying $\mathfrak{D}_{\omega^k}$ to $\mathfrak{D}_{\omega^{k-1}}\cdots\mathfrak{D}_{\omega}f$ and finally arrive at
		\begin{align*}
			\left\|\sum_{ |\supp(\alpha)|=\ell} \tau_\alpha^{(\omega)}\cdots\tau_\alpha^{(\omega^{K-1})} a_{\alpha}z^{\alpha}\right\|_{\Omega_K^n}
			\le (2+2\sqrt{2})^{(K-1)\ell}\|f\|_{\Omega_K^n}.
		\end{align*}
		We claim that for any $\alpha$ with $|\supp(\alpha)|=\ell$, the cumulative factor introduced by iterating $\mathfrak{D}_{\omega^k}$'s,
		\begin{equation*}
			\tau_K(\ell):=\prod_{1\le k\le K-1}\tau_\alpha^{(\omega^k)}=	\prod_{1\le k\le K-1}\prod_{j:\alpha_j\neq 0}\left(1-(\omega^k)^{\alpha_j}\right),
		\end{equation*}
		is a nonzero constant depending only on $K$ and $\ell=|\supp(\alpha)|$.
		In fact, it will suffice to argue that
		\begin{equation}\label{eq:c prime}
			\prod_{1\le k\le K-1}\left(1-(\omega^k)^j\right)=:d_K
		\end{equation}
		is some nonzero constant independent of $1\le j\le K-1$.
		To see this is sufficient, recall $m_j(\alpha)=|\{k:\alpha_k=j\}|,1\le j\le K-1$ are such that
		\begin{equation*}
			\sum_{1\le j\le K-1}m_j(\alpha)=|\supp(\alpha)|=\ell.
		\end{equation*}
		Then if \eqref{eq:c prime} holds,
		we have 
		\begin{align*}
			\prod_{1\le k\le K-1}\prod_{j:\alpha_j\neq 0}\left(1-(\omega^k)^{\alpha_j}\right)
			&= \prod_{1\le k\le K-1}\prod_{1\le j\le K-1}\left(1-(\omega^k)^{j}\right)^{m_j(\alpha)}\\
			&= \prod_{1\le j\le K-1}\left[\prod_{1\le k\le K-1}\left(1-(\omega^k)^{j}\right)\right]^{m_j(\alpha)}\\
			&= \prod_{1\le j\le K-1}d_K^{m_j(\alpha)}\\
			&= d_K^\ell
		\end{align*}
		and thus the claim is shown with $\tau_K(\ell)=d_K^\ell$.
		This would yield
		\begin{align*}
			\left\| \sum_{ |\supp(\alpha)|=\ell} a_{\alpha} z^{\alpha}\right\|_{\Omega_K^n}
			\le |d_K|^{-\ell}(2+2\sqrt{2})^{(K-1)\ell}\|f\|_{\Omega_K^n},
		\end{align*}
		as desired.

		Now it remains to verify \eqref{eq:c prime}.
		We remark that until this step we have not used the assumption that $K>2$ is prime.
		Note that \eqref{eq:c prime} is nonzero since otherwise $\omega^{jk}=1$; \emph{i.e.}, $K\mid jk$ for some $1\le k\le K-1$, and this is not possible since $K>2$ is prime. Moreover, again by the primality of $K$, 
		\begin{equation*}
			\{jk: 1\le k\le K-1\}\equiv \{k:1\le k\le K-1\} \mod K, ~1\le j\le K-1.
		\end{equation*}
		Therefore, 
		\begin{equation*}
			\prod_{1\le k\le K-1}\left(1-(\omega^{k})^j\right)=\prod_{1\le k\le K-1}(1-\omega^{k})=:d_K
		\end{equation*}
		is independent of $ 1\le j\le K-1$, which finishes the proof of the claim. 
		
		Now we have shown that for the maximal support-homogeneous part $f_\ell$ of $f$:
		\begin{equation}\label{ineq:top ell}
			\|f_\ell\|_{\Om_K^n}\le C_K^d\|f\|_{\Omega_K^n}
		\end{equation}
		with $C_K=(2+2\sqrt{2})^{K-1}/|d_K|$.
		Repeating the same argument to $f-f_\ell$ whose maximal support-homogeneous part is $f_{\ell-1}$, we get by triangle inequality and \eqref{ineq:top ell} that
		\begin{equation*}
			\|f_{\ell-1}\|_{\Om_K^n}\;\leq\; C_K^d\|f-f_\ell\|_{\Om_K^n}\;
			\overset{\eqref{ineq:top ell}}{\leq}\; C_K^d(1+C_K^{d})\|f\|_{\Om_K^n}.
		\end{equation*}
		Iterating this procedure, we may obtain for all $ 0\le k\le \ell$ that
		\begin{equation}\label{ineq:top ell-k}
			\|f_{\ell-k}\|_{\Om_K^n}\le C_K^d(1+C_K^{d})^{k}\|f\|_{\Omega_K^n}
			=[(1+C_K^d)^{k+1}-(1+C_K^d)^{k}]\|f\|_{\Omega_K^n}.
		\end{equation}
		In fact, we have shown that \eqref{ineq:top ell-k} holds for $k=0,1$. Assume that \eqref{ineq:top ell-k} holds for $0\le j\le k-1$ and let us prove \eqref{ineq:top ell-k} for $k$. Since $f_{\ell-k}$ is the maximal support-homogeneous part of $f-\sum_{0\le j\le k-1}f_{\ell-j}$, we have by the previous argument proving \eqref{ineq:top ell}, the triangle inequality and the induction assumption that
		\begin{align*}
			\|f_{\ell-k}\|_{\Om_K^n}
			\; \le \; & C_K^d\left\|f-\sum_{0\le j\le k-1}f_{\ell-j}\right\|_{\Omega_K^n}\\
			\; \le \; &C_K^d\left(\|f\|_{\Omega_K^n}+\sum_{0\le j\le k-1}\|f_{\ell-j}\|_{\Omega_K^n}\right)\\
			\;\le \;&C_K^d\left(1+\sum_{0\le j\le k-1}\big[(1+C_K^d)^{j+1}-(1+C_K^d)^{j}\big]\right)\|f\|_{\Omega_K^n}\\
			\;=\; &C_K^d(1+C_K^d)^{k}\|f\|_{\Omega_K^n}.
		\end{align*}
		This finishes the proof of \eqref{ineq:top ell-k}. In particular for all $0\le j\le \ell$
		\begin{equation}
            \label{eq:induction-done}
			\|f_{j}\|_{\Om_K^n}\le C_K^d(1+C_K^{d})^{\ell-j}\|f\|_{\Omega_K^n}
			\le C_K^d(1+C_K^d)^{d}\|f\|_{\Omega_K^n}
			\le (2C^2_K)^{d^2}\|f\|_{\Omega_K^n},
		\end{equation}
		which completes the proof of the lemma.
	\end{proof}
	
	\begin{proof}[Proof of Lemma \ref{lem:splitting-lemma} for any $K>2$]
		As above, let $\ell\le d$ be the largest support size of monomials in $f$. For any $m\le \ell$, we denote by $f_m$ the $m$-support homogeneous part of $f$.
		Now, instead of choosing different $\xi$ to homogenize the accumulated factors for each monomial, we shall repeat the same procedure with the fixed choice of $\xi:=\omega$.
		For convenience $\mathfrak{D} f$ will denote $\mathfrak{D}_{\omega}f$.
		Concretely, we consider
		\begin{equation*}
			\mathfrak{D} f(z)=\sum_{ |\supp(\alpha)|=\ell} a_{\alpha}\tau_\alpha^{(\omega)} z^{\alpha}
		\end{equation*}
		and apply $\mathfrak{D}$ again to obtain
		\begin{align*}
			\left\|\mathfrak{D}^2 f\right\|_{\Omega_K^n}
			\le (2+2\sqrt{2})^{\ell}\|\mathfrak{D} f\|_{\Omega_K^n}
			\le (2+2\sqrt{2})^{2\ell}\|f\|_{\Omega_K^n}
		\end{align*}
		as before.
		Inductively, one obtains 
		\begin{align}\label{ineq:c(alpha)^k}
			\mathfrak{D}^{k}f&= \sum_{ |\supp(\alpha)|=\ell} a_{\alpha}(\tau_\alpha^{(\omega)})^k z^{\alpha}                                       \notag
			\\
			\text{with}\qquad\left\| \mathfrak{D}^{k}f\right\|_{\Omega_K^n}&\le (2+2\sqrt{2})^{k\ell}\|f\|_{\Omega_K^n}, \qquad k\ge 1.          
		\end{align}
		
		Now we group the terms $a_\alpha z^\alpha$ leading to the same factors $\tau_\alpha^{(\omega)}$ into polynomials $h_j(z)$.
		(Recall that $\tau_\alpha^{(\omega)}$ might coincide for distinct $\alpha$'s, see \eqref{ex:same-taus}).
		Define the polynomials $h_j(z)$ such that
		\begin{equation*}
			f_{\ell}(z)=\sum_{ |\supp(\alpha)|=\ell} a_{\alpha} z^{\alpha}=\sum_{1\le j\le M}h_j(z), \quad\text{and}
		\end{equation*}
		\begin{equation*}
			\mathfrak{D}^{k}f(z)=\sum_{ |\supp(\alpha)|=\ell} a_{\alpha}(\tau_\alpha^{(\omega)})^k z^{\alpha}=\sum_{1\le j\le M}c_j^k h_j(z),\qquad k\ge 1.
		\end{equation*}
		Here, $\{c_j, 1\le j\le M\}$ are pairwise distinct nonzero constants of the form
		\begin{equation*}
			c_j=\tau_\alpha^{(\omega)}=\prod_{1\le j\le n:\alpha_j\neq 0}\left(1-\omega^{\alpha_j}\right)
			=\prod_{1\le j\le K-1}\left(1-\omega^j\right)^{m_j(\alpha)},
		\end{equation*}
		for some $\alpha$ (or multiple $\alpha$'s) and
		where again $m_j(\alpha)=|\{k:\alpha_k=j\}|$.
		The last equality holds since  $\alpha_j\le K-1$. Recall that $$\sum_{1\le j\le K-1}m_j(\alpha)=\ell \le d,$$ 
		so the number of different $c_j$'s, which we denoted by $M$, is bounded by 
		\begin{equation*}
			\left|\left\{(m_1(\alpha),\dots, m_{K-1}(\alpha))\in \mathbf{Z}_{\ge 0}^{K-1}:\sum_{1\le j\le K-1}m_j(\alpha)\le d\right\}\right|
			\le \binom{K-1+d}{d}
		\end{equation*}
		which depends only on $d$ and $K$. According to \eqref{ineq:c(alpha)^k}, we have
		\begin{equation*}
			\begin{pmatrix}
				\mathfrak{D}f\\
				\mathfrak{D}^{2}\!f\\
				\vdots \\
				\mathfrak{D}^{M}\!f\\
			\end{pmatrix}
			=\underbrace{\begin{pmatrix}
					c_1 & c_2 & \cdots & c_M\\
					c_1^2 & c_2^2 &\cdots & c_M^2\\
					\vdots & \vdots & \ddots & \vdots\\
					c_1^M & c_2^M & \cdots & c_M^M
			\end{pmatrix}}_{=:\; V}
			\begin{pmatrix}
				h_1\\
				h_2\\
				\vdots \\
				h_M\\
			\end{pmatrix}.
		\end{equation*}
		
		The $M\times M$ modified Vandermonde matrix $V$ has determinant
		\[\det(V) = \left(\prod_{j=1}^M c_j\right) \left( \prod_{1\leq j < k \leq M}(c_j-c_k)\right).\]
		Since $c_j$'s are distinct and nonzero we have $\det(V)\neq 0$. So $V$ is invertible and thus 
		\begin{align*}
			f_{\ell}=&\sum_{1\le j\le M}h_j=
			\begin{pmatrix}
				1&1&\cdots &1
			\end{pmatrix}
			\begin{pmatrix}
				h_1\\
				h_2\\
				\vdots \\
				h_M\\
			\end{pmatrix}\\
			=&	\begin{pmatrix}
				1&1&\cdots &1
			\end{pmatrix}
			V^{-1}
			\begin{pmatrix}
				\mathfrak{D}^{1}\!f\\
				\mathfrak{D}^{2}\!f\\
				\vdots \\
				\mathfrak{D}^{M}\!f\\
			\end{pmatrix}
			=:\sum_{1\le j\le M}\eta_j\, \mathfrak{D}^{j}\!f.
		\end{align*}
		Here $\eta:=(\eta_1,\dots, \eta_M)=\begin{pmatrix}
			1&1&\cdots &1
		\end{pmatrix}
		V^{-1}$ is an explicit $M$-dimensional vector depending only on $d$ and $K$. Therefore, 
		\begin{equation*}
			\|f_{\ell}\|_{\Om_K^n}\le \sum_{1\le j\le M} |\eta_j|\,	\|\mathfrak{D}^{j}\!f\|_{\Om_K^n}
			\le \|\eta\|_1(2+\sqrt{2})^{Md}	\|f\|_{\Om_K^n},
		\end{equation*}
		where we used \eqref{ineq:c(alpha)^k} in the last inequality. The constant $\|\eta\|_1({2+2\sqrt{2}})^{Md}$ is dimension-free and depends only on $d$ and $K$.
		
		This finishes the proof for $f_\ell$. 
        The proof for general $0\le j\le \ell$ follows inductively by the same argument as in the  prime $K>2$ case (via inequalities \eqref{ineq:top ell}, \eqref{ineq:top ell-k}, and \eqref{eq:induction-done}).
	\end{proof}
	

\section{Applications}
\label{sect:applications}

In this section, we briefly mention several applications and leave detailed investigation to the future. 

\subsection{Sidon constants}
\label{subsec:sidon}

Bohnenblust--Hille inequalities can be used to estimate the \emph{Sidon constants} \cite{DFOOS,DGMMM}.
Fix $K\ge 3$ and $n\ge 1$. Recall that for any  $\mathscr{S}\subset \{0,1,\dots, K-1\}^n$, the \emph{Sidon constant} of $\mathscr{S}$ (or the characters $(z^\alpha)_{\alpha\in \mathscr{S}}$ of the cyclic group $ \{0,1,\dots, K-1\}^n$), denoted by $\Sidon(\mathscr{S})$, 
is defined to be the best constant $C>0$ such that 
\begin{equation*}
	\sum_{\alpha\in\mathscr{S}}|a_\alpha|\le C\left\|\sum_{\alpha\in\mathscr{S}}a_{\alpha}z^\alpha\right\|_{\Omega_K^n}
\end{equation*}
for all functions $f=\sum_{\alpha\in\mathscr{S}}a_{\alpha}z^\alpha$. The Bohnenblust--Hille inequalities \eqref{ineq:bh cyclic} give an upper bound of the Sidon constants for 
\begin{equation*}
	\mathscr{S}_{K}(n,\le d):=\left\{ \alpha=(\alpha_1,\dots, \alpha_n)\in\{0,1,\dots, K-1\}^{n}:\sum_{j=1}^{n}\alpha_j\le d \right\}.
\end{equation*}

The next proposition is a direct analog of results in \cite{OOS,DFOOS}.

\begin{prop}\label{prop:sidon}
	For any $K\ge 3,n\ge 1$ and $d\ge 1$, we have 
	\begin{equation}
		\Sidon(	\mathscr{S}_{K}(n,\le d))\le C(d,K)\left[\sum_{0\le k\le d}\binom{n+k}{n}\right]^{\frac{d-1}{2d}},
	\end{equation}
for some $C(d,K)>0$ depending only on $d$ and $K$. We may choose $C(d,K)=C_K^{d^2}$ when $K$ is prime. 
\end{prop}

\begin{proof}
	For any $f=\sum_{\alpha\in	\mathscr{S}_{K}(n,\le d)}a_{\alpha}z^\alpha$ on $\Omega_K^n$, we have by Bohnenblust--Hille inequalities \eqref{ineq:bh cyclic}  and H\"older's inequality that 
\begin{align*}
	\sum_{\alpha\in	\mathscr{S}_{K}(n,\le d)}|a_\alpha|
	&\le
	\left(\sum_{\alpha\in	\mathscr{S}_{K}(n,\le d)}|a_\alpha|^{\frac{2d}{d+1}}\right)^{\frac{d+1}{2d}}\left(\sum_{\alpha\in	\mathscr{S}_{K}(n,\le d)}1\right)^{\frac{d-1}{2d}}\\
	&\le \BH^{\le d}_{\Om_K}|	\mathscr{S}_{K}(n,\le d)|^{\frac{d-1}{2d}}\left\|f\right\|_{\Omega_K^n},
\end{align*}
where $\BH^{\le d}_{\Omega_K}$ is the best Bohnenblust--Hille constant in \eqref{ineq:bh cyclic}.
Note that 
\begin{equation*}
	|\mathscr{S}_{K}(n,\le d)|
	\le \left|\left\{ \alpha\in \mathbb{Z}_{\ge 0}^{n}:\sum_{j=1}^{n}\alpha_j\le d \right\}\right|
	\le \sum_{0\le k\le d}\binom{n+k}{n},
\end{equation*}
then we finish the proof by Theorem \ref{thm:bh cyclic}.
\end{proof}

\subsection{Bohr's radius}
\label{subsect:bohr radius}

The Bohr's radius for a class $\mathcal{F}$ of functions $f:\Omega_K^n\to \mathbf{C}$ is defined as 
\begin{equation*}
	\rho(\mathcal{F}):=\sup\left\{\rho>0: \sum_{\alpha} |a_{\alpha}| \rho^{|\alpha|}\le \|f\|_{\Omega_K^n}~\textnormal{ for all }~f=\sum_{\alpha}a_{\alpha}z^{\alpha}\in\mathcal{F}\right\}.
\end{equation*}
Consider the following classes of functions 
\begin{equation*}
	\mathcal{F}^{n,K}_{=d}:=\{\textnormal{all $d$-homogeneous polynomials on $\Omega_K^n$}\}.
\end{equation*}

\begin{prop}
	For any $K\ge 3,n\ge 1$ and $d\ge 1$, we have 
	\begin{equation*}
	\rho(\mathcal{F}^{n,K}_{=d})\ge\left[\BH^{\le d}_{\Om_K}\binom{n+d}{n}^{\frac{d-1}{2d}}\right]^{-1/d}.
	\end{equation*}
\end{prop}

\begin{proof}
The proof is again a combination of Bohnenblust--Hille inequalities and H\"older's inequality. In the homogeneous case, we actually have the relation
\begin{equation}
\rho(\mathcal{F}^{n,K}_{=d})=\Sidon(	\mathscr{S}_{K}(n,= d))^{-1/d}
\end{equation}
by definition. Then the desired result follows from the proof of Proposition \ref{prop:sidon}.
\end{proof}
\subsection{Juntas and learning bounded low-degree polynomials}
\label{subsect:juntas}

Fix $n\ge 1$ and $K\ge 3$. For any $k\le n$, we say that a polynomial $f:\Omega_K^n\to \mathbf{C}$ is a \emph{$k$-junta} if it depends on at most $k$ coordinates. 

The next result states that bounded low-degree polynomials are close to some juntas. We denote by $\|f\|_p$ the $L^p$-norm of $f$ with respect to the uniform probability measure on $\Omega_K^n$.

\begin{prop}\label{prop:junta}
	Fix  $K\ge 3,n\ge 1$ and $d\ge 1$. Suppose that $f:\Omega_K^n\to \mathbf{C}$ is of degree at most $d$ and $\|f\|_{\Omega_K^n}\le 1$. Then for any $\epsilon>0$, there exists a $k$-junta $g:\Omega_K^n\to \mathbf{C}$ such that 
	\begin{equation*}
		\|f-g\|_2\le \epsilon\qquad\textnormal{with}\qquad k\le \frac{d\left(\BH^{\le d}_{\Omega_K}\right)^{2d}}{\epsilon^{2d}}.
	\end{equation*}
Moreover, we may choose $k\le \epsilon^{-2d}d C_K^{d^3}$ if $K$ is prime.
\end{prop}

\begin{proof}
	Suppose that $f$ has the Fourier expansion 
	\begin{equation*}
		f(\alpha)=\sum_{|\alpha|\le d}a_{\alpha} z^{\alpha}
	\end{equation*}
where $\alpha\in \{0,1,\dots, K-1\}^{n}$ as in \eqref{eq:fourier expansion cyclic}. For any $\lambda >0$, consider 
	\begin{equation*}
		A_{\lambda}:=\left\{\alpha:|a_{\alpha}|>\lambda \right\}.
	\end{equation*}
Then by Markov's inequality, cyclic BH inequalities \eqref{ineq:bh cyclic} and the assumption that $|f|\le 1$, the cardinality $|A_\lambda|$ of $A_\lambda$ satisfies
\begin{equation*}
	|A_{\lambda}|\le \lambda^{-\frac{2d}{d+1}}\sum_{|\alpha|\le d}|a_{\alpha}|^{\frac{2d}{d+1}}
	\le \lambda^{-\frac{2d}{d+1}}\left(\BH^{\le d}_{\Omega_K}\right)^{\frac{2d}{d+1}}.
\end{equation*}
Take $g_\lambda:=\sum_{\alpha\in A_{\lambda}}a_{\alpha} z^{\alpha}$. 
By Plancherel's identity, we have 
\begin{equation*}
	\|f-g_\lambda\|_2^2=\sum_{\alpha\notin A_{\lambda}}|a_{\alpha}|^2\le \lambda^{\frac{2}{d+1}}\sum_{|\alpha|\le d}|a_{\alpha}|^{\frac{2d}{d+1}}
	\le	 \lambda^{\frac{2}{d+1}}\left(\BH^{\le d}_{\Omega_K}\right)^{\frac{2d}{d+1}}.
\end{equation*}
Choosing $\lambda=\epsilon^{d+1}\left(\BH^{\le d}_{\Omega_K}\right)^{-d}$ so that 
\begin{equation*}
		 \lambda^{\frac{2}{d+1}}\left(\BH^{\le d}_{\Omega_K}\right)^{\frac{2d}{d+1}}
		 =\epsilon^2,
\end{equation*}
we obtain $	\|f-g_\lambda\|_2\le \epsilon$. Moreover, $g_\lambda$ is a $k$-junta with 
\begin{equation*}
	k\le  d \lambda^{-\frac{2d}{d+1}}\left(\BH^{\le d}_{\Omega_K}\right)^{\frac{2d}{d+1}}
	\le \epsilon^{-2d}d\left(\BH^{\le d}_{\Omega_K}\right)^{2d}.
\end{equation*}
This finishes the proof, in view of Theorem \ref{thm:bh cyclic}.
\end{proof}

The following theorem is an extension of the main result in \cite{EI22}. 

\begin{prop}\label{prop:learning}
Fix  $K\ge 3,n\ge 1$ and $d\ge 1$. Suppose that $f:\Omega_K^n\to \mathbf{C}$ is of degree at most $d$ and $\|f\|_{\Omega_K^n}\le 1$. Fix $\epsilon,\delta\in (0,1)$ and 
\begin{equation*}
	M\ge \frac{e^5 d(d+1) \left(\BH^{\le d}_{\Omega_K}\right)^{2d}}{\epsilon^{d+1}}\log\left(\frac{2Kn}{\delta}\right).
\end{equation*}
Then for any $M$ i.i.d. random variables $X_1,\dots X_M$ uniformly distributed on $\Omega_K^n$, as well as the queries 
\begin{equation*}
	\left(X_j, f(X_j)\right),\qquad 1\le j\le M,
\end{equation*}
we can construct a random polynomial $\widetilde{f}:\Omega_K^n\to \mathbf{C}$ such that 
\begin{equation*}
	\|f-\widetilde{f}\|_2^2\le \epsilon
\end{equation*} 
with probability at least $1-\delta$.
\end{prop}

\begin{proof}
	Fix $b>0$. Suppose that  $f$ has the Fourier expansion 
	\begin{equation*}
		f(\alpha)=\sum_{|\alpha|\le d}a_{\alpha} z^{\alpha}
	\end{equation*}
	as in \eqref{eq:fourier expansion cyclic}. For any $|\alpha|\le d$, consider 
	\begin{equation*}
		\widetilde{a}_\alpha:=\frac{1}{M_b}\sum_{j=1}^{M_b}f(X_j)\overline{X_j}^\alpha,
	\end{equation*}
with $M_b>0$ to be chosen later. 
Recall that for $z=(z_1,\dots, z_n)\in \Omega_K^n$, $\overline{z}^{\alpha}=\overline{z}_1^{\alpha_1}\cdots \overline{z}_n^{\alpha_n}$. Since $\widetilde{a}_{\alpha}$ is the sum of bounded (noting that $|f(X_j)\overline{X_j}^{\alpha}|\le 1$) i.i.d. random variables with $\E [\widetilde{a}_{\alpha}]=a_{\alpha}$, we have
by Chernoff bound that 
\begin{equation*}
	\Pr\{|\widetilde{a}_\alpha-a_{\alpha}|>b\}\le 2e^{-M_b b^2/2},\qquad |\alpha|\le d.
\end{equation*}
Then the union bound gives 
\begin{align*}
		\Pr\{|\widetilde{a}_\alpha-a_{\alpha}|\le b, |\alpha|\le d\}
		&\ge 1-2e^{-M_b b^2/2}|\{\alpha: |\alpha|\le d\}|\\
		&\ge 1-2e^{-M_b b^2/2}\sum_{0\le k\le d}\binom{n+k}{n}\\
		&\ge 1-\delta
\end{align*}
where the last inequality holds if we choose 
\begin{equation}\label{eq:Mb}
	M_b=\left\lceil \frac{2}{b^2}\log \left(\frac{2}{\delta}\sum_{0\le k\le d}\binom{n+k}{n}\right)\right\rceil.
\end{equation}
Fix $a\in (b,\infty)$ and consider 
\begin{equation*}
	\mathcal{S}_a:=\{\alpha:|\widetilde{a}_\alpha|>a\}.
\end{equation*}
Then with probability at least $1-\delta$, we have 
\begin{equation}\label{ineq:cases}
	\begin{cases*}
		|a_\alpha|	\le |\widetilde{a}_\alpha|+|a_\alpha-\widetilde{a}_\alpha|\le a+b&$\alpha\notin \mathcal{S}_a$\\
		|a_\alpha|			\ge |\widetilde{a}_\alpha|-|a_\alpha-\widetilde{a}_\alpha|\ge a-b&$\alpha\in \mathcal{S}_a$
	\end{cases*}.
\end{equation}
Consider the random function $h_{a,b}:\Om_K^n\to \mathbf{C}$ given by
\begin{equation*}
	h_{a,b}(z)=\sum_{\alpha\in \mathcal{S}_a}\widetilde{a}_{\alpha} z^{\alpha}.
\end{equation*}
Similar to the estimates in the proof of Proposition \ref{prop:junta}, we have with probability at least $1-\delta$
\begin{equation*}
	|\mathcal{S}_a|\le (a-b)^{-\frac{2d}{d+1}}\sum_{\alpha\in\mathcal{S}_a}|a_\alpha|^{\frac{2d}{d+1}}\le (a-b)^{-\frac{2d}{d+1}}\left(\BH^{\le d}_{\Omega_K}\right)^{\frac{2d}{d+1}}
\end{equation*}
 by \eqref{ineq:cases}.
Then with probability at least $1-\delta$, one has 
\begin{align*}
	\|f-h_{a,b}\|^2_2
	=&\sum_{\alpha\in\mathcal{S}_a}|a_\alpha-\widetilde{a}_\alpha|^2+\sum_{\alpha\notin\mathcal{S}_a}|a_\alpha|^2\\
	\le& |\mathcal{S}_a|b^2+(a+b)^{\frac{2}{d+1}}\sum_{\alpha\notin\mathcal{S}_a}|a_\alpha|^{\frac{2d}{d+1}}\\
	\le &\left(\BH^{\le d}_{\Omega_K}\right)^{\frac{2d}{d+1}}\left((a-b)^{-\frac{2d}{d+1}}b^2+(a+b)^{\frac{2}{d+1}}\right).
\end{align*}
Choosing $a=b(1+\sqrt{d+1})$ and arguing as in \cite{EI22}, we have 
\begin{equation*}
		\|f-h_{a,b}\|^2_2\le \epsilon, \qquad\textnormal{whenever}\qquad b^2\le e^{-5}d^{-1}\epsilon^{d+1}\left(\BH^{\le d}_{\Omega_K}\right)^{-2d}.
\end{equation*}
Plugging this in \eqref{eq:Mb}, we may choose 
\begin{equation*}
	M_b= \left\lceil \frac{e^5 d \left(\BH^{\le d}_{\Omega_K}\right)^{2d}}{\epsilon^{d+1}}\log\left(\frac{2}{\delta}\sum_{0\le k\le d}\binom{n+k}{n}\right)\right\rceil
\end{equation*}
 so that $	\|f-h_{a,b}\|^2_2\le \epsilon$ with probability at least $1-\delta$. Using the estimate (recalling that $d\le (K-1)n$)
 \begin{equation*}
 	\sum_{0\le k\le d}\binom{n+k}{n}
 	\le 	\sum_{0\le k\le d}(n+k)^k
 	\le 	\sum_{0\le k\le d}(Kn)^k
 	\le (Kn)^{d+1}
 \end{equation*} 
we finish the proof with the desired $M=M_b$.
\end{proof}
\subsection{Noncommutative Bohnenblust--Hille inequality for Heisenberg--Weyl matrix basis}
\label{HWbasis}


Fix $K> 2$ and $\om=e^{\frac{2\pi i}{K}}$. Let $\mathbb{Z}_K=\{0,1,\dots, K-1\}$ be the additive cyclic group of order $K$ and 
$$\{e_j:j\in \mathbb{Z}_K\}=\{e_j:0\le j\le K-1\}$$ 
be the standard basis of $\C^K$. The  ``shift" operator $X$ and ``phase" operator $Z$ are defined via 
\begin{equation*}
	X e_j= e_{j+1},\qquad Ze_j = \om^j e_j,\qquad \textnormal{for all} \qquad j\in \mathbb{Z}_K.
\end{equation*}
Note that $X^K=Z^K=\1$. See more in \cite{AEHK}. In the following, everything is $\!\!\mod\, K$ unless otherwise stated.

Below we consider Heisenberg--Weyl collection of matrices of size $K\times K$:
$$
\textnormal{HW}(K):=\{X^\ell Z^m\}_{\ell,m\in \mathbb{Z}_K}\,.
$$
These are unitary matrices and form a basis of the $K\times K$ complex matrix algebra $M_K(\C)$. Moreover, they are orthonormal with respect to
the normalized trace $\tr_K=K^{-1}\tr$.

Fix $n\ge 1$. Any HW observable $A\in M_K(\C)^{\otimes n}$ has a unique Fourier expansion with respect to $\textnormal{HW}(K)$:
\begin{equation*}
	A=\sum_{\vec{\ell},\vec{m}\in \mathbb{Z}_K^n}\widehat{A}(\vec{\ell},\vec{m})X^{\ell_1}Z^{m_1}\otimes \cdots \otimes X^{\ell_n}Z^{m_n},
\end{equation*}
where $\widehat{A}(\vec{\ell},\vec{m})\in\C$ is the Fourier coefficient at $(\vec{\ell},\vec{m})$. 
We say that $A$ is \emph{of degree-$d$} if $\widehat{A}(\vec{\ell},\vec{m})=0$ whenever 
\begin{equation*}
	|(\vec{\ell},\vec{m})|:=\sum_{j=1}^{n}(\ell_j+m_j)>d.
\end{equation*}
Here, $0\le \ell_j,m_j\le K-1$ and we do not $\!\!\mod K$ freely. 


We denote by $\widehat{A}$ the sequence of Fourier coefficients of $A$, and write 
\begin{equation*}
	\|\widehat{A}\|_p:=\left(\sum_{\vec{\ell},\vec{m}\in \mathbb{Z}_K^n}|\widehat{A}(\vec{\ell},\vec{m})|^p\right)^{1/p},\qquad 1\le p<\infty.
\end{equation*}

Our main result for the Heisenberg--Weyl basis is the following quantum analog of Bohnenblust--Hille inequality:

\begin{theorem}
\label{thm:bh HW}
	Fix a prime number $K> 2$ and suppose $d\ge 1$. 
	Then the Bohnenblust--Hille inequalities hold for the Heisenberg--Weyl basis:
	for any $n\ge 1$ and any $A\in M_K(\C)^{\otimes n}$ of degree-$d$, we have 
	\begin{equation*}
		\|\widehat{A}\|_{\frac{2d}{d+1}}\le C(d,K)\|A\|,
	\end{equation*}
	with $C(d,K)\le (K+1)^d \textnormal{BH}_{\Om_K}^{\le d}$. In particular, we may choose $C(d,K)=C_K^{d^2}$ for some constant $C_K>0$ depending only on $K$.
\end{theorem}

As the statement suggests, we may actually reduce the problem to the Bohnenblust--Hille inequality for cyclic groups $\Om_K^{n}, n\ge 1$, via the reduction proved in an earlier version of \cite{SVZ}. In this reduction step, we need $K$ to be prime at the time of writing. Then Theorem \ref{thm:bh HW} follows immediately from the reduction and Theorem \ref{thm:bh cyclic}. Here we assume the prime $K>2$. The $K=2$ case is contained in \cite{VZ22}. Namely the Bohnenblust--Hille inequalities in the quantum case for $K=2$ (Pauli matrices as basis) can be reduced to the classical case for $K=2$ (Boolean cube case \eqref{ineq:bh boolean})  up to an additional factor of $(K+1)^d=3^d$. In this case, the classical Bohnenblust--Hille constant has a better bound $\BH^{\le d}_{\Omega_2}\le C^{\sqrt{d\log d}}$, so the quantum constant has a better estimate of $3^dC^{\sqrt{d\log d}}\le C^d$ for some universal $C>0$.

	\subsection{Learning bounded quantum observables of low-degrees}
	\label{learn}
	
	The  reader may consult an earlier version of \cite{SVZ} to learn how quantum  Bohnenblust--Hille inequalities in Theorem \ref{thm:bh HW} can be used in learning bounded quantum observables $A\in M_K(\mathbf{C})^{\otimes n}$ that are of low degree in the Heisenberg--Weyl basis (we also presented the statements for another basis named after Gell-Mann using the corresponding Bohnenblust--Hille type inequalities).
	
	Also, one may use noncommutative Bohnenblust--Hille inequalities to obtain other applications, such as junta-type results aforementioned. See \cite{VZ22,SVZ} for details. 

	\appendix
\section{Retracing the proofs of Bohnenblust--Hille inequality}
\label{sect:appendix}
%

	In this appendix we examine the difficulties of adapting the proof of the Bohnenblust--Hille inequalities for $\T^n$ and $\Omega_2^n$ in \cite{BPS,DMP} to $\Omega_K^n$ for $K>2$.
	Ultimately this motivates Conjecture \ref{conjectuer}, which in turn receives a partial answer in Appendix \ref{sect:partial}.

	We will see the main obstacle to adapting the standard BH argument lies in the \emph{polarization step}. Since we only aim to illustrate the barrier in the proof scheme here, we will restrict ourselves to a subfamily of degree-$d$ polynomials as they capture the difficulty.
	
	The class of polynomials under consideration in this section is the family of homogeneous polynomials of degree at most $d$ with individual degree at most $K/2.$ That is, we fix $K>2$ and $f:\Omega_K^n\to \mathbf{C}$ such that 
	\begin{equation*}
		f(z)=\sum_{|\alpha|=d}a_{\alpha}z^{\alpha}
	\end{equation*}
	where in the summation $\alpha=(\alpha_1,\dots, \alpha_n)\in \{0,1,\dots, [\frac{K}{2}]\}^n$. To each such $\alpha$, we associate a unique $\bi=(i_1,\dots i_d)$ with $1\le i_1\le\cdots \le i_d\le n$ that is defined by
	\begin{equation*}
		\bi=(1,\stackrel{\alpha_1}{\dots}, 1,2,\stackrel{\alpha_2}{\dots}, 2, \dots, n,\stackrel{\alpha_n}{\dots}, n),
	\end{equation*}
	or equivalently, 
	\begin{equation*}
		\alpha_j:=|\{k:i_k=j\}|.
	\end{equation*}
	Put for fixed $K$ and $n$ that 
	\begin{equation*}
		J(d):=\left\{\bi=(i_1,\dots, i_d):1\le i_1\le \dots\le i_d\le n, \max_k|\{k:i_k=j\}|\le [K/2]\right\}.
	\end{equation*}
	So our function $f:\Omega_K^n\to \mathbf{C}$ in consideration can be represented as 
	\begin{equation}\label{eq:fourier_exp_d_homo}
		f(z)=\sum_{\bi\in J(d)}a_{\bi}z_{\bi}
		=\sum_{\bi\in J(d)}a_{i_1,\dots, i_d}z_{i_1}\cdots z_{i_d},
	\end{equation}
	with $z_{\bi}:=z_{i_1}\cdots z_{i_d}$, 
	and we aim to show
	\begin{equation*}
		\left(\sum_{\bi \in J(d)}|a_{\bi}|^{\frac{2d}{d+1}}\right)^{\frac{d+1}{2d}}
		\le B_K(d)\left\|\sum_{\bi\in J(d)}a_{\bi}z_{\bi}\right\|_{\Omega_K^n}
	\end{equation*}
	with $B_K(d)>0$ being the best constant in this restricted case. 
	
	We retrace the proof of \cite{BPS,DMP} in several steps and will need more notation. 
	For $\bi,\bj\in [n]^d$, the notation $\bi\sim \bj$ means that there is a permutation $\sigma$ of $[d]$ such that $i_{\sigma(k)}=j_{k},k\in [d]$. For any $\bi$, we denote by $[\bi]$ the equivalent class of all $\bj\in [n]^d$ such that $\bj\sim\bi$ and $\mult(\bi)$  the cardinality of $[\bi]$.
	Then 
	\begin{align*}
	\sum_{\bi\in J(d)}|a_{\bi}|^{\frac{2d}{d+1}}
	&=	\sum_{\bi\in [n]^d} \mult(\bi)^{-\frac{1}{d+1}} \left(\frac{|a_{[\bi]}|}{\mult(\bi)^{1/2}}\right)^{\frac{2d}{d+1}}\\
	&\le \sum_{\bi\in [n]^d} \left(\frac{|a_{[\bi]}|}{\mult(\bi)^{1/2}}\right)^{\frac{2d}{d+1}}
	\end{align*}
which is our initial step, \textbf{Step 0}.

	\medskip
	\noindent\textbf{Step 1.} By the so-called \emph{Blei's inequality} \cite{BPS,DMP}, we have for any $1\le k\le d$,
	\begin{equation*}
	 \sum_{\bi\in [n]^d} \left(\frac{|a_{[\bi]}|}{\mult(\bi)^{1/2}}\right)^{\frac{2d}{d+1}}
		\le \prod_{S\subset [d]:|S|=k}\left[\sum_{\bi_S}\left(\sum_{\bi_{\bar{S}}}\frac{|a_{[\bi]}|^{2}}{\mult(\bi)}\right)^{\frac{1}{2}\cdot \frac{2k}{k+1}}\right]^{\frac{k+1}{2k}\cdot \binom{d}{k}^{-1}}
	\end{equation*}
	where for each $S\subset [d]$, we used the convention that $\bar{S}=[d]\setminus S$ and 
	$$\sum_{\bi_S}=\sum_{i_j:j\in S}.$$
	The above inequality is a consequence of interpolation (or H\"older's inequality) and Minkowski's inequality.
	
	\medskip
	\noindent\textbf{Step 2.} For each fixed $S\subset [d]$ with $|S|=k$, note that 
	\begin{equation*}
		\frac{\mult(\bi)}{\mult(\bi_{\bar{S}})}\le \mult(\bi)\le d!.
	\end{equation*}
So 
\begin{equation*}
	\frac{|a_{[\bi]}|^{2}}{\mult(\bi)}\le d!\frac{|a_{[\bi]}|^{2}}{\mult(\bi)^2}\mult(\bi_{\bar{S}})
\end{equation*}
and thus we have shown that 
\begin{flalign*}
			\mathrlap{\left(\sum_{\bi \in J(d)}|a_{\bi}|^{\frac{2d}{d+1}}\right)^{\frac{d+1}{2d}}} &\\
			&&{} \le \sqrt{d!}
			\prod_{S\subset [d]:|S|=k}\left[\sum_{\bi_S}\left(\sum_{\bi_{\bar{S}}}\frac{|a_{[\bi]}|^{2}}{\mult(\bi)^2}\mult(\bi_{\bar{S}})\right)^{\frac{1}{2}\cdot \frac{2k}{k+1}}\right]^{\frac{k+1}{2k}\cdot \binom{d}{k}^{-1}}.
\end{flalign*}
For each $\bi$ with $S$ fixed, consider the polynomial
\begin{equation*}
	P_{\bi_S}(x) :=\sum_{\bi_{\bar{S}}\in [n]^{d-k}}\frac{a_{[\bi]}}{\mult(\bi)}x_{\bi_{\bar{S}}},\qquad x=(x_1,\dots, x_n)\in 	\Omega_K^n,
\end{equation*}
whose Fourier expansion is given by 
\begin{equation*}
		P_{\bi_S}(x) =\sum_{\bi_{\bar{S}}\in J(d-k)}\frac{a_{[\bi]}}{\mult(\bi)}\mult(\bi_{\bar{S}})x_{\bi_{\bar{S}}}.
\end{equation*}
	 Thus Plancherel's identity gives 
	\begin{align*}
	\E_{x\sim \Omega_K^n}|P_{\bi_S}(x)|^2
	=&\E_{x\sim \Omega_K^n} \left|\sum_{\bi_{\bar{S}}\in [n]^{d-k}}\frac{a_{[\bi]}}{\mult(\bi)}x_{\bi_{\bar{S}}}\right|^2\\
	=	&\sum_{\bi_{\bar{S}}\in J(d-k)}\frac{|a_{[\bi]}|^2}{\mult(\bi)^2}\mult(\bi_{\bar{S}})^2\\
	=&\sum_{\bi_{\bar{S}}\in [n]^{d-k}}\frac{|a_{[\bi]}|^2}{\mult(\bi)^2}\mult(\bi_{\bar{S}})
	\end{align*}
	where $\E_{x\sim \Omega_K^n}$ is with respect to the Haar measure on $\Omega_K^n$. So we just proved
 
	\begin{flalign*}
		\mathrlap{\left(\sum_{\bi \in J(d)}|a_{\bi}|^{\frac{2d}{d+1}}\right)^{\frac{d+1}{2d}}} &\\
		&&{} \le \sqrt{d!}
		\prod_{S\subset [d]:|S|=k}\left[\sum_{\bi_S}\left(
		\E_{x\sim \Omega_K^n}|P_{\bi_S}(x)|^2 \right)^{\frac{1}{2}\cdot \frac{2k}{k+1}}\right]^{\frac{k+1}{2k}\cdot \binom{d}{k}^{-1}}.
	\end{flalign*}

	\medskip
	\noindent\textbf{Step 3.} Now we are in a position to apply the following moment comparison estimates: for polynomials $g:\Omega_K^n\to \mathbf{C}$ of degree at most $d$ we have 
	\begin{equation}\label{ineq:moment comparison}
		\left(\E_{z\sim \Omega_K^n}|g(z)|^2\right)^{1/2}\le \rho_p^{Cd}\left(\E_{z\sim \Omega_K^n}|g(z)|^p\right)^{1/p},\qquad 1\le p<2,
	\end{equation}
	where $\rho_p=(p-1)^{-1/2}$ for $1<p<2$ and $\rho_1=e$, together with $C>0$ being some universal constant. In fact, it is a standard argument to derive this moment comparison estimate from the (not necessarily sharp) hypercontractivity of the semigroup that takes $z^{\alpha}$ to $e^{-t|\alpha|}z^{\alpha}$. Such hypercontractivity estimates have been known for cyclic $\Omega_K^n,K>2$ (see e.g. \cite{A,JPPP}).
    We remark that in these references the definition of $|\cdot|$ is not our choice. The $L_p$-$L_2$ hypercontractivity estimates in their choice of $|\cdot|$ are stronger. 
    
    
	
	Thus, we continue with our estimates and arrive at 
	\begin{equation*}
		\sum_{\bi_S}\left(
		\E_{x\sim \Omega_K^n}|P_{\bi_S}(x)|^2 \right)^{\frac{1}{2}\cdot \frac{2k}{k+1}}
		\le \rho_{2k/(k+1)}^{2C(d-k)k/(k+1)}\E_{x\sim \Omega_K^n} \sum_{\bi_S}\left|P_{\bi_S}(x)\right|^{\frac{2k}{k+1}}.
	\end{equation*}
	So we obtain 
		\begin{align*}
		\left(\sum_{\bi \in J(d)}|a_{\bi}|^{\frac{2d}{d+1}}\right)^{\frac{d+1}{2d}}
		\le C(d,k)
		\prod_{S\subset [d]:|S|=k}\left[
		\E_{x\sim \Omega_K^n} \sum_{\bi_S}\left|P_{\bi_S}(x)\right|^{\frac{2k}{k+1}}
		\right]^{\frac{k+1}{2k}\cdot \binom{d}{k}^{-1}}.
	\end{align*}
	
	\medskip
	\noindent\textbf{Step 4.} Now for fixed $x\in \Omega_K^n$ we apply the Bohnenblust--Hille inequalities for $k$:
	\begin{equation*}
		\sum_{\bi_S}\left|P_{\bi_S}(x)\right|^{\frac{2k}{k+1}}
		\le B_K(k)^{\frac{2k}{k+1}}\sup_{y\in \Omega_K^{n}}\left|\sum_{\bi_S}\sum_{\bi_{\bar{S}}}\frac{a_{[\bi]}}{\mult(\bi)}x_{\bi_{\bar{S}}}y_{\bi_{S}}\right|^{\frac{2k}{k+1}}.
	\end{equation*}
	So for all $S\subset [d]$ with $|S|=k$, one has
	\begin{equation*}
\left[
\E_{x\sim \Omega_K^n} \sum_{\bi_S}\left|P_{\bi_S}(x)\right|^{\frac{2k}{k+1}}
\right]^{\frac{k+1}{2k}}
		\le B_K(k)\sup_{x\in \Omega_K^n, y\in \Omega_K^{n}}\left|\sum_{\bi_S}\sum_{\bi_{\bar{S}}}\frac{a_{[\bi]}}{\mult(\bi)}x_{\bi_{\bar{S}}}y_{\bi_{S}}\right|.
	\end{equation*}

	\medskip
	\noindent\textbf{Step 5.}
	All combined, we have shown that  
	\begin{equation*}
		\left(\sum_{\bi\in J(d)}|a_{\bi}|^{\frac{2d}{d+1}}\right)^{\frac{d+1}{2d}}
		\le C(d,k)B_K(k)\sup_{|S|=k}\sup_{x\in \Omega_K^n, y\in \Omega_K^{n}}\left|\sum_{\bi_S}\sum_{\bi_{\bar{S}}}\frac{a_{[\bi]}}{\mult(\bi)}x_{\bi_{\bar{S}}}y_{\bi_{S}}\right|.
	\end{equation*}
	We are ready to discuss the main obstacle of polarization. 
	For this let us recall the polarization result in \cite{DMP} adapted to our setting. Consider any homogeneous polynomial on $\mathbf{C}^n$ of degree $d$
	$$P(z)=\sum_{\bi\in J(d)}b_{\bi}z_{\bi}=\sum_{1\le i_1\le \dots\le i_d\le n}b_{i_1,\dots, i_d}z_{i_1}\cdots z_{i_d},$$
	and define $L_P:(\mathbf{T}^n)^d\to \mathbf{C}$
	$$L_P(z^{(1)},\dots, z^{(d)})=
	\sum_{\bi\in [d]^n}\frac{b_{[\bi]}}{\mult(\bi)}z^{(1)}_{i_1}\cdots z^{(d)}_{i_d}.$$
	Then $L_P$ is $d$-affine, i.e., $L_P$ is affine in each $z^{(j)}, j\in [d]$. It is also symmetric; \emph{i.e.}, $L_P$ is invariant under the permutation of coordinates $z^{(1)},\dots, z^{(d)}$. Moreover, 
	on the diagonal one has
	\begin{equation*}
		L_P(z,\dots, z)=\sum_{\bi\in [d]^n}\frac{b_{[\bi]}}{\mult(\bi)}z_{i_1}\cdots z_{i_d}=P(z).
	\end{equation*}
	Thus for any $z,w\in \mathbf{T}^n$ and $t\in [0,1]$
	\begin{align*}
		p_{z,w}(t):=&P\left(\frac{1+t}{2}z+\frac{1-t}{2}w\right)\\
		=&\sum_{k=0}^{d}\binom{d}{k}\left(\frac{1+t}{2}\right)^{k}\left(\frac{1-t}{2}\right)^{d-k}L_P(\underbrace{z,\dots,z}_{k},\underbrace{w,\dots, w}_{d-k}).
	\end{align*}
	By  \cite[Proposition 4]{DMP}, we have 
	\begin{align*}
		\left|L_P(\underbrace{z,\dots,z}_{k},\underbrace{w,\dots, w}_{d-k})\right|
		&\le C(d,k)\sup_{t\in [0,1]}|p_{z,w}(t)|\\
		&=C(d,k) \sup_{t\in [0,1]}\left|P\left(\frac{1+t}{2}z+\frac{1-t}{2}w\right)\right|.
	\end{align*}
	Note that for any $z,w\in \Omega_K$ and for any $t\in [0,1]$, $\frac{1+t}{2}z+\frac{1-t}{2}w$ belongs to the convex hull of $\Omega_K$, $\conv(\Omega_K)$. So
	\begin{flalign*}
		\mathrlap{\sup_{z,w\in \Omega_K^n}\left|L_P(\underbrace{z,\dots,z}_{k},\underbrace{w,\dots, w}_{d-k})\right|}\\
		& &&\le C(d,k)\sup_{z,w\in \Omega_K^n} \sup_{t\in [0,1]}\left|P\left(\frac{1+t}{2}z+\frac{1-t}{2}w\right)\right|\\
		& &&\le C(d,k)\sup_{z\in \conv(\Omega_K)^n}\left|P\left(z\right)\right|.
	\end{flalign*}
	Now we extend our polynomial $f:\Omega_K^n\to \mathbf{C}$ to $\mathbf{C}^n$, still denoted by $f$, via the same Fourier expansion:
	\begin{equation*}
		f(z)=\sum_{\bi\in J(d)}a_{\bi}z_{\bi},\qquad z\in \mathbf{C}^n.
	\end{equation*}
	Apply the above polarization result to $P=f$ to obtain
	\begin{equation*}
		\sup_{|S|=k}\sup_{x\in \Omega_K^n, y\in \Omega_K^{n}}\left|\sum_{\bi_S}\sum_{\bi_{\bar{S}}}\frac{a_{[\bi]}}{\mult(\bi)}x_{\bi_{\bar{S}}}y_{\bi_{S}}\right|
		\le C(d,k)\sup_{z\in \conv(\Omega_K)^n}\left|f\left(z\right)\right|.
	\end{equation*}
	
	\medskip
	
	These are essentially the proof ingredients of Bohnenblust--Hille inequalities for $\T^n=\Omega_\infty^n$ and $\Omega_2^n$ and  we conclude with
	\begin{equation}
	\label{ineq:convex}
		\left(\sum_{\bi\in J(d)}|a_{\bi}|^{\frac{2d}{d+1}}\right)^{\frac{d+1}{2d}}
		\le C(d,k)B_K(k)\sup_{z\in \conv(\Omega_K)^n}\left|f\left(z\right)\right|.
	\end{equation}
	If we could prove  \eqref{ineq:mmp} which we quote here
	\begin{equation}\label{ineq-appendix:mmp}
		\sup_{z\in \conv(\Omega_K)^n}\left|f(z)\right|
		\le C(d,K)\sup_{z\in \Omega_K^n}\left|f(z)\right|,
	\end{equation}
	then we will deduce 
	\begin{equation*}
		B_K(d)\le C(d,k,K)B_K(k)
	\end{equation*}
	and derive the upper bound of $B_K(d)$ inductively. 
	
	\medskip
	
	The main difficulty we have to overcome is the absence of \eqref{ineq-appendix:mmp} which is obvious for $\T^n$ and $\Om_2^n=\{-1, 1\}^n$ as explained in Section \ref{sect:obstacle}.
	
	\medskip

Therefore, it seems that the best that one can deduce from the above argument is the convex hull version of Bohnenblust--Hille inequalities for cyclic groups $\Omega_K^n,K>2$:
	\begin{equation}\label{ineq in appendix:convex hull bh}
\left(		\sum_{|\alpha|=d}|a_{\alpha}|^{\frac{2d}{d+1}}\right)^{\frac{d+1}{2d}}
\le C(d,K)\|f\|_{\conv(\Omega_K)^n}
	\end{equation}
	for all $f(z)=\sum_{|\alpha|=d}a_{\alpha}z^\alpha$, which motivates Conjecture \ref{conjectuer}.

\begin{rem}
Denote $\BH_{\conv(\Omega_K)}^{=d}$ the best constant of the convex hull version of Bohnenblust--Hille inequality \eqref{ineq in appendix:convex hull bh} for homogeneous polynomials on $\conv(\Omega_K)^n$ of degree $d$.
Following the similar argument as above, one arrives at 
\begin{equation}\label{ineq:inductive ineq convex hull}
\BH_{\conv(\Omega_K)}^{=d}\le C(d,k,K)\BH_{\conv(\Omega_K)}^{=k}
\end{equation}
for some $C(d,k,K)<\infty$. In fact, the only difference is to apply \eqref{ineq in appendix:convex hull bh} in \textbf{Step 4} instead of the standard Bohnenblust--Hille inequality for $\Omega_K^n$. The \textbf{Step 5} remains the same, which gives \eqref{ineq:inductive ineq convex hull}. We leave it to the interested readers to chase the upper bound of $\BH_{\conv(\Omega_K)}^{=d}$ following the arguments in \cite{DMP}.
\end{rem}

\section{A partial answer to Conjecture \ref{conjectuer}: convex hull of rows of the discrete Fourier transform matrices}
\label{sect:partial}
This subsection collects some partial results about Conjecture \ref{conjectuer}; it also appeared in an earlier version of \cite{SVZ}.
We first show that the conjecture cannot hold with $C(d,K)=1$ in the one-dimensional case, implying one cannot obtain the desired result simply by tensorization without using the low-degree assumption.
Then we provide an affirmative answer under a certain constraint, recovering the Bohnenblust--Hille inequalities for cyclic groups in a partial case by this method.

\subsection{Constant cannot be $1$} 
\label{not1}

 Let $K=3$ and $\omega:=e^{\frac{2\pi i}{3}}$. Consider the polynomial 
\begin{equation*}
	p(z):=p(1)\frac{(z-\omega)(z-\omega^2)}{(1-\omega)(1-\omega^2)}
	+p(\omega)\frac{(z-1)(z-\omega^2)}{(\omega-1)(\omega-\omega^2)}
	+p(\omega^2)\frac{(z-1)(z-\omega)}{(\omega^2-1)(\omega^2-\omega)}
\end{equation*}
with $p(1), p(\omega), p(\omega^2)$ to be chosen later. Put $z_0:=\frac{1+\omega}{2}\in \conv(\Omega_3)$. Then 
\begin{equation*}
	|z_0-1|=|z_0-\omega|=\frac{\sqrt{3}}{2},\qquad |z_0-\omega^2|=\frac{3}{2}.
\end{equation*}
Now we choose $p(1), p(\omega), p(\omega^2)$ to be complex numbers of modules $1$ such that 
\begin{equation*}
	p(1)\frac{(z_0-\omega)(z_0-\omega^2)}{(1-\omega)(1-\omega^2)}=\left|\frac{(z_0-\omega)(z_0-\omega^2)}{(1-\omega)(1-\omega^2)}\right|
	=\frac{\frac{3\sqrt{3}}{4}}{3}
	=\frac{\sqrt{3}}{4},
\end{equation*}
\begin{equation*}
	p(\omega)\frac{(z_0-1)(z_0-\omega^2)}{(\omega-1)(\omega-\omega^2)}
	=\left|\frac{(z_0-1)(z_0-\omega^2)}{(\omega-1)(\omega-\omega^2)}\right|
	=\frac{\frac{3\sqrt{3}}{4}}{3}
	=\frac{\sqrt{3}}{4},
\end{equation*}
\begin{equation*}
	p(\omega^2)\frac{(z_0-1)(z_0-\omega)}{(\omega^2-1)(\omega^2-\omega)}
	=\left|\frac{(z_0-1)(z_0-\omega)}{(\omega^2-1)(\omega^2-\omega)}\right|
	=\frac{\frac{3}{4}}{3}
	=\frac{1}{4}.
\end{equation*}
Therefore, this choice of $p$ satisfies 
\begin{equation*}
	\|p\|_{\conv(\Omega_3)}\ge |p(z_0)|
	=\frac{\sqrt{3}}{4}+\frac{\sqrt{3}}{4}+\frac{1}{4}
	=\frac{1+2\sqrt{3}}{4}>1=\|p\|_{\Omega_3}.
\end{equation*}

\subsection{A partial solution}
In this part, we give a partial solution to Conjecture \ref{conjectuer}. 
Let us remark that for notational convenience, in the remaining part we will actually consider cyclic groups of order $2K-1$ instead of $K$.
A similar argument works for cyclic groups of even order.
We start with the following key lemma saying that 
\begin{equation*}
	(1,z,z^2\dots, z^{K-1})\in \mathbf{C}^K
\end{equation*}
lies in the convex hull of ``first half'' of row vectors of the $K$-point discrete Fourier transform matrix as long as $|z|$ is small enough. 

\begin{lemma}
	\label{key lem}
	Fix $K\ge 2$, put $\xi_k=e^{\frac{2\pi i k}{2K-1}}$. There exists $\eps_0=\eps_0(K)\in (0,1)$ such that, for all $z\in\bC$ with $|z|\le \eps_0$, one can find $p_k=p_k(z)\ge  0, 0\le k\le 2K-2$ satisfying 
	\begin{equation}\label{eq:complex equations to solve}
		z^m=\sum_{k=0}^{2K-2}p_k\xi_k^{m}, \qquad 0\le m\le K-1.
	\end{equation}
	In particular, when $m=0$, $\sum_{k=0}^{2K-2}p_k=1$.
\end{lemma}

\begin{proof}
	Put $\theta=\frac{2\pi}{2K-1}$. The equations \eqref{eq:complex equations to solve} are equivalent to ($p_k$'s are non-negative and thus real)
	\begin{equation}\label{eq:real equations to solve}
		\begin{cases*}
			\sum_{k=0}^{2K-2}p_k=1&\\
			\sum_{k=0}^{2K-2}p_k\cos(km\theta)=\Re z^m& $1\le m\le K-1$\\
			\sum_{k=0}^{2K-2}p_k\sin(km\theta)=\Im z^m& $1\le m\le K-1$
		\end{cases*}.
	\end{equation}
	Or equivalently, we want to find a solution to $D_K\vec p=\vec v_z$ with each entry of $\vec p$ being non-negative. Here $D_K$ is a $(2K-1)\times (2K-1)$ real matrix given by
	\begin{equation*}
		D_K=
		\begin{bmatrix*}
			1 & 1 & 1 &\cdots & 1\\
			1 & \cos(\theta) &\cos(2\theta)&\cdots &\cos((2K-2)\theta)\\
			\vdots & \vdots & \vdots &&\vdots \\
			1 & \cos((K-1)\theta) &\cos(2(K-1)\theta)&\cdots &\cos((2K-2)(K-1)\theta)\\
			1 & \sin(\theta) &\sin(2\theta)&\cdots &\sin((2K-2)\theta)\\
			\vdots & \vdots & \vdots &&\vdots \\
			1 & \sin((K-1)\theta) &\sin(2(K-1)\theta)&\cdots &\sin((2K-2)(K-1)\theta)
		\end{bmatrix*},
	\end{equation*}
	and $\vec v_z=(1,\Re z, \dots, \Re z^{K-1},\Im z,\dots, \Im z^{K-1})^T\in \bR^{2K-1}$.
	
	Note first that $D_K$ is non-singular. In fact, assume that $D_K\vec x=\vec 0$ with $\vec x=(x_0,x_1, \dots, x_{2K-2})^T\in \bR^{2K-1}$. Then 
	\begin{equation*}
		\sum_{k=0}^{2K-2}x_k\xi_k^{m}=0, \qquad 0\le m\le K-1.
	\end{equation*}
	Since each $x_k$ is real and $\xi^{2K-1}=1$, we have by taking conjugation that 
	\begin{equation*}
		\sum_{k=0}^{2K-2}x_k\xi_k^{m}=0, \qquad K\le m\le 2K-1.
	\end{equation*}
	Altogether we get 
	\begin{equation*}
		\sum_{k=0}^{2K-2}x_k\xi_k^{m}=0, \qquad 0\le m\le 2K-2.
	\end{equation*}
	Since the Vandermonde matrix associated to $(1,\xi,\dots, \xi_{2K-2})$ has determinant
	\begin{equation*}
		\prod_{0\le j<k\le 2K-2}(\xi_{j}-\xi_{k})\neq 0,
	\end{equation*}
	we get $\vec x=\vec 0$. So $D_K$ is non-singular. 
	
	Therefore, for any $z\in\bC$, the solution to \eqref{eq:real equations to solve}, thus to \eqref{eq:complex equations to solve}, is given by 
	$$
	\vec p_z=(p_0(z),p_1(z),\dots, p_{2K-2}(z))=D_K^{-1}\vec v_z\in \bR^{2K-1}.
	$$ 
	
	Notice one more thing about the rows of $D_K$. As
	$$
	\sum_{k=0}^{2K-2} \xi_k^m =0,\quad \forall m=1, 2, \dots, 2K-2,
	$$
	we have automatically that vector $\vec v_0 := (\frac1{2K-1}, \dots, \frac1{2K-1})$ of length $2K-1$ gives
	$$
	D_K \vec v_0 =(1, 0, 0, \dots, 0)^T\,.
	$$
	
	For any $k$-by-$k$ matrix $A$ denote 
	$$\|A\|_{\infty\to\infty}:=\sup_{\vec 0\neq \vec v\in\bR^{k}}\frac{\|A\vec v\|_\infty}{\|\vec v\|_\infty}.$$
	So we have 
	\begin{align*}
		\|\vec p_z-\vec p_0\|_\infty
		\le &\|D_K^{-1}\|_{\infty\to \infty}\|\vec v_z-\vec v_0\|_\infty\\
		=&\|D_K^{-1}\|_{\infty\to \infty}\max\left\{\max_{1\le k\le K-1}|\Re z^k|,\max_{1\le k\le K-1}|\Im z^k|\right\}\\
		\le &\|D_K^{-1}\|_{\infty\to \infty}\max\left\{|z|,|z|^{K-1}\right\}.
	\end{align*}
	That is, 
	\begin{equation*}
		\max_{0\le j\le 2K-2}\left|p_j(z)-\frac{1}{2K-1}\right|
		\le \|D_K^{-1}\|_{\infty\to \infty}\max\left\{|z|,|z|^{K-1}\right\}.
	\end{equation*}
	Since $D_K^{-1}\vec v_0=\vec p_0$, we have $\|D_K^{-1}\|_{\infty\to \infty}\ge 2K-1$.
	Put
	$$
	\eps_0:=\frac{1}{(2K-1)\|D_K^{-1}\|_{\infty\to \infty}}\in \left(0,\frac{1}{(2K-1)^2}\right].
	$$
	Thus, whenever $|z|<\eps_0<1$, we have 
	\begin{equation*}
		\max_{0\le j\le 2K-2}\left|p_j(z)-\frac{1}{2K-1}\right|
		\le \eps_0 \|D_K^{-1}\|_{\infty\to \infty}\le \frac{1}{2K-1}.
	\end{equation*}
	Therefore, $p_j(z)\ge 0$ for all $0\le j\le 2K-2$ and the proof is complete.
\end{proof}

With Lemma \ref{key lem}, we may give a partial solution to Conjecture \ref{conjectuer}. Suppose that 
\[f(z):=\sum_\al a_\al z^\al,\qquad  z=(z_1, \dots, z_n)\in \bC^n\]
is any analytic  polynomial of $n$ complex variables of degree at most $d$
and such that in each variable $z_i$ its degree is at most $K-1$. We shall prove that 
	\begin{equation}\label{ineq:partial conj}
	\max_{z\in \conv(\Omega_{2K-1})^n}|f(z)|\le 
	C(d,K)\max_{z\in \Omega_{2K-1}^n}|f(z)|.
\end{equation}
This gives a positive answer to Conjecture \ref{conjectuer} for odd $K$ with an additional constraint that the individual degree is not larger than $\frac{K-1}{2}$.

To prove \eqref{ineq:partial conj}, denote $\xi:=e^{\frac{2\pi i}{2K-1}}$ and $\xi_k:=\xi^k$. Note that by Lemma \ref{key lem}, there exists $\eps_0=\eps_0(K)\in (0,1)$ such that for all $ z=(z_1,\dots, z_n)$ with $\| z\|_\infty\le \eps_0$, we have 
	\begin{equation*}
		z^m_j=\sum_{k=0}^{2K-2}p^{(j)}_k \xi^{m}_k, \qquad 1\le j\le n, \quad 0\le m\le K-1,
	\end{equation*}
	where $p_k^{(j)}=p_k^{(j)}(z_j)>0$ satisfies $\sum_{k=0}^{2K-2}p_k^{(j)}=1$ for any $1\le j\le n$. Then 
	%
	we have 
	\begin{align*}
		\left|f(z_1,\dots, z_n)\right|
		&=\left|\sum_{\al_1,\dots, \al_n}a_{\al_1,\dots, \al_n}z_1^{\al_1}\cdots z_n^{\al_n}\right|\\
		&=\left|\sum_{\al_1,\dots, \al_n}\sum_{k_1,\dots, k_n=0}^{2K-2}a_{\al_1,\dots, \al_n}p_{k_1}^{(1)}\cdots p_{k_n}^{(n)}\xi_{k_1}^{\al_1}\cdots \xi_{k_n}^{\al_n}\right|\\
		&\le \sum_{k_1,\dots, k_n=0}^{2K-2}p_{k_1}^{(1)}\cdots p_{k_n}^{(n)}\left|\sum_{\al_1,\dots, \al_n}a_{\al_1,\dots, \al_n}\xi_{k_1}^{\al_1}\cdots \xi_{k_n}^{\al_n}\right|\\
		&= \sum_{k_1,\dots, k_n=0}^{2K-2}p_{k_1}^{(1)}\cdots p_{k_n}^{(n)}|f(\xi_{k_1},\dots, \xi_{k_n})|\\
		&\le \sum_{k_1,\dots, k_n=0}^{2K-2}p_{k_1}^{(1)}\cdots p_{k_n}^{(n)}\sup_{z \in \Om_{2K-1}^n}\left|f(z)\right|\\
		&=\sup_{z \in \Om_{2K-1}^n}\left|f( z)\right|.
	\end{align*}
	So we have shown that 
	\begin{equation}\label{ineq:key}
		\sup_{\|z\|_\infty \le \eps_0}\left|f(z)\right|
		\le \sup_{z \in \Om_{2K-1}^n}\left|f(z)\right|.
	\end{equation}
	Now consider 
	$$
	P(z):=f(\eps_0 z_1,\dots, \eps_0 z_n)=\sum_{\al}\eps_0^{|\al|}a_{\al}z^{\al}.
	$$
	Then
	\begin{equation}
	\label{Pf}
		 \sup_{\|z\|_\infty\le 1} |P (z)| \le \sup_{\|z\|_\infty\le \eps_0} |f(z)| \le \sup_{z \in \Om_{2K-1}^n}\left|f(z)\right|\,.
		 \end{equation}
For analytic polynomial $P$ on $\T^n$ its $k$-homogeneous parts $P_k, 0\le k\le d$ are bounded trivially:
\begin{equation}\label{ineq:cauchy estimate}
	 \sup_{\|z\|_\infty\le 1} |P_k (z)| \le  \sup_{\|z\|_\infty\le 1} |P (z)| \,.
\end{equation}

	But $P_k(z)= \eps_0^k f_k(z)$, where $f_k$ is the $k$-homogeneous part of $f$. Combining \eqref{Pf} and \eqref{ineq:cauchy estimate}, we get
	$$
	 \sup_{\|z\|_\infty\le 1} |f_k (z)|  \le \eps_0^{-k}  \sup_{z \in \Om_{2K-1}^n}\left|f(z)\right| \,.
	 $$
	 
	 Summation in $k\le d$ and the triangle inequality finish the estimate
	 $$
	 	 \sup_{z\in \conv(\Om_{2K-1})^n} |f (z)| \le
	 \sup_{\|z\|_\infty\le 1} |f (z)|  \le (d+1) \eps_0^{-d}  \sup_{z \in \Om_{2K-1}^n}\left|f(z)\right| \, ,
	 $$
	 which gives our generalized maximum modulus principle \eqref{ineq:mmp}.
	 \begin{rem}
	 \label{N3}
	 As we mentioned earlier, the Conjecture \ref{conjectuer} has been resolved in our later work \cite{SVZgmp} with completely different arguments. 
\end{rem}

\end{document}